\documentclass[11pt]{article}
\usepackage{amsfonts}
\usepackage{graphicx}
\usepackage{amsmath,amssymb}

\usepackage{amsfonts,mathrsfs,graphicx,latexsym,bm}
\usepackage{amsmath,amsthm,latexsym,amssymb,amsxtra}
\usepackage[pagebackref,driverfallback=dvipdfmx]{hyperref}
\usepackage{verbatim}
\usepackage{enumerate}

\newcommand{\be}{\begin{equation}}
\newcommand{\ee}{\end{equation}}

\def\dot{\cdot}
\newtheorem{theorem}{\bf Theorem}[section]{\bf}{\it}
\newtheorem{lemma}{\bf Lemma}[section]{\bf}{\it}
\newtheorem{defi}{\bf Definition}[section]{\bf}{\it}
{\bf}{\it}
\newtheorem{coro}{Corollary}[section]{\bf}{\it}
\newtheorem{remark}{Remark}[section]{\bf}{\it}

{\it}{\rm}
\newcommand{\bprf}{\begin{proof}}
\newcommand{\eprf}{\end{proof}}

\newcommand{\bprfp}[1]{\begin{proof}[Proof of Proposition \nopunct] \ref{#1}. \ \  }
\newcommand{\bprfl}[1]{\begin{proof}[Proof of Lemma \nopunct] \ref{#1}. \ \  }
\newcommand{\bprfc}[1]{\begin{proof}[Proof of Corollary \nopunct] \ref{#1}. \ \  }
\newcommand{\bprft}[1]{\begin{proof}[Proof of Theorem \nopunct] \ref{#1}. \ \  }

\begin{document}
%\begin{CJK*}{GBK}{song}
\begin{center}
{\Large \bf Holomorphic Curves into Algebraic Varieties Intersecting Divisors in Subgeneral Position}
\end{center}

\begin{center}
{\large Qingchun Ji\footnotemark[1],\ \ \ \ Qiming Yan\footnotemark[2]\ \ \ \ Guangsheng Yu\footnotemark[3]}
\end{center}

%\author{
%	 Qingchun Ji
%	 \footnotemark[1],
%	 Qiming Yan
%	 \footnotemark[2],
%	 Guangsheng Yu
%	 \footnotemark[3]
%}

{\bf Abstract:} Recently, there are many developments on the second main theorem for holomorphic curves into algebraic varieties intersecting divisors in general position or subgeneral position. In this paper, we refine the concept
of subgeneral position by introducing the notion of the index of subgeneral position.  With this new notion we give some surprising  improvement of the previous known second main theorem type results.
%In doing so, we introduce the filtration method (so the joint-filtration lemma can be used) working on the (arbitrary) projective variety $X$ by extending the original CZ-method which only works on the case $X$ being a complete section. The ``filtration method" %whose main idea comes from Cherry-Dethloff-Tan also gives a new proof of Ru's result (see the proof of Theorem 1.1).
Moreover, via the analogue between Nevanlinna theory and Diophantine approximation, the corresponding Schmidt's subspace type theorems are also established in the final section.



\footnotetext[1]{School of Mathematical Sciences, Fudan University, Shanghai Center for Mathematical Sciences, Shanghai 200433, P.R. China;}
\footnotetext[2]{School of Mathematical Sciences, Tongji University, Shanghai 200092, P.R. China;}
\footnotetext[3]{Department of Mathematics, University of Shanghai for Science and Technology, Shanghai 200093, P.R. China.}
%\footnotetext{\hspace*{6mm}{
%	
%		School of Mathematical Sciences and Key Laboratory of Mathematics for Nonlinear Sciences, Fudan University, Shanghai 200433, P.R. China;\\ School of Mathematical Sciences, Tongji University, Shanghai 200092, P.R. China;\\ Department of Mathematics, University of Shanghai for Science and Technology, Shanghai 200093, P.R. China.}\\
%	\hspace*{5mm}
\footnotetext{Email Address: qingchunji@fudan.edu.cn; yan$\underline{\mbox{ }}$qiming@hotmail.com; ygsh@usst.edu.cn.\\
		\hspace*{5mm}{Q. Ji and G. Yu was partially supported by NSFC11671090. Q. Yan was partially supported by NSFC11571256.}}%}

\section{Introduction}

In higher-dimensional Nevanlinna theory, Ru \cite{1,3} established second main theorems for algebraically non-degenerate holomorphic curves into complex projective spaces or complex projective varieties (may be singular) intersecting hypersurfaces in general position, which solved a long-standing conjecture by Shiffman \cite{2}. On the other hand, the case of hypersurfaces in subgeneral position was also considered in the last few years. To state some of the results, we recall the following notions.

\begin{defi} Let $D_1,\ldots,D_q$ be effective (Weil or Cartier) divisors on a projective variety $X$.

(i) $D_1,\ldots,D_q$ are said to be {\bf in general position} if for any subset $I\subset\{1,\ldots,q\}$ with $\sharp I\le \dim X+1$,
$$
{\rm codim} \bigcap_{i\in I}{\rm Supp}D_i\ge \sharp I.
$$

(ii) $D_1,\ldots,D_q$ are said to be in $m$-{\bf subgeneral position} if for any subset $I\subset\{1,\ldots,q\}$ with $\sharp I\le m+1$,
$$
\dim \bigcap_{i\in I}{\rm Supp}D_i\le m-\sharp I.
$$
(Here we set $\dim \emptyset=-1$.)
\end{defi}

It is easy to see that the divisors are in general position if they are in $(\dim X)$-subgeneral position. Note that when $m<\dim X$, the notion of $m$-subgeneral position is stronger than general position, so this paper focuses mainly on the case $m\geq\dim X$.
The following is a reformulated version of Ru's theorem in \cite{3}, for the  notations see section \ref{sec2}.

\noindent{\bf Theorem A.}\quad Let $X$ be a complex projective variety of dimension $n\ge 1$. Let $D_1,\ldots,D_q$ be effective divisors on $X$, located in general position on $X$. Suppose that there exists an ample divisor $A$ on $X$ and positive integers $d_j$ such that $D_j\sim d_jA$ (i.e., $D_j$ is linearly equivalent to $d_jA$) for $j=1,\ldots,q$. Let $f:{\mathbb{C}}\rightarrow X$ be an algebraically non-degenerate holomorphic curve, which means that the image of $f$ is Zariski dense. Then, for every $\varepsilon>0$,
\begin{eqnarray*}
\|\sum_{j=1}^q\frac{1}{d_j}m_f(r,D_j)\le (n+1+\varepsilon)T_{f,A}(r),
\end{eqnarray*}
where $``\|"$ means the estimate holds for all large $r$ outside a
set of finite Lebesgue measure.

%(Here we use the standard notations in Nevanlinna theory which will be introduced in the next section.)

In \cite{4}, Chen, Ru and Yan considered the case of subgeneral position.

\noindent{\bf Theorem B.} (Theorem 1.1 in \cite{4}, reformulated)\quad  Let $X$ be a complex projective variety of dimension $n\ge 1$. Let $D_1,\ldots,D_q$ be effective divisors on $X$, located in $m$-subgeneral position on $X$. Suppose that there exists an ample divisor $A$ on $X$ and positive integers $d_j$ such that $D_j\sim d_jA$ for $j=1,\ldots,q$. Let $f:{\mathbb{C}}\rightarrow X$ be an algebraically non-degenerate holomorphic curve. Then, for every $\varepsilon>0$,
\begin{eqnarray}
\|\sum_{j=1}^q\frac{1}{d_j}m_f(r,D_j)\le (m(n+1)+\varepsilon)T_{f,A}(r).\label{eq1-1}
\end{eqnarray}

Note that the bound $m(n+1)$ in (\ref{eq1-1}) is not optimal. Moreover, when $m=n$, Theorem B cannot recover Theorem A. When $m>n$, motivated by the work of Levin \cite{5}, Shi and Ru \cite{6} improved (\ref{eq1-1}) as
\begin{eqnarray}
\|\sum_{j=1}^q\frac{1}{d_j}m_f(r,D_j)\le \left(\frac{m(m-1)}{m+n-2}(n+1)+\varepsilon\right)T_{f,A}(r).\label{eq1-2}
\end{eqnarray}
Recently, Si \cite{6.5} improved the bound in (\ref{eq1-1}) and (\ref{eq1-2}) to $(m-n+1)(n+1)$, which recovers Theorem A when $m=n$. However, this bound is also far from sharp.

In this paper, we refine the concept of subgeneral position by introducing the notion
of the index of subgeneral position and obtain some second main theorems
based on this notion. %A similar notion was introduced \cite{JY} to establish second main theorems for differentiably non-degenrate holomorphic maps without the normal crossing assumption.

\begin{defi}
	Let $D_1,\ldots,D_q$ be effective (Weil or Cartier) divisors on a projective variety $X$ of dimension $n$. Let $m\ge n$ and ${\kappa}\le n$ be two positive integers. We say $D_1,\ldots,D_q$ are in $m$-{\bf subgeneral position with index} ${\kappa}$ if $D_1,\ldots,D_q$ are in $m$-subgeneral position and
 for any subset $J\subset\{1,\ldots,q\}$ with $\sharp J\le{\kappa}$,
	$$
	{\rm codim} \bigcap_{j\in J}{\rm Supp}D_j\ge \sharp J.
	$$
\end{defi}

Obviously, the index $\kappa$ is at least one for any divisors in subgeneral position. If $\kappa$ is strictly greater than one, the following theorem indicates that this notion is not vacuous.

\begin{theorem}\label{thm1.4}

Let $X$ be a complex projective variety with $\dim X=n$, and let $D_1,\ldots,D_q$ be effective Cartier divisors in $m$-subgeneral position with index $\kappa(>1)$ on $X$. Suppose that there exists an ample divisor $A$ on $X$ and positive integers $d_j$ such that $D_j\sim d_jA$ for $j=1,\ldots,q$. Let $f:{\mathbb{C}}\rightarrow X$ be an algebraically non-degenerate holomorphic curve. Then, for every $\varepsilon>0$,
	\begin{eqnarray}
		\|\sum_{j=1}^q\frac{1}{d_j}m_f(r,D_j)\le \left(\max\left\{\frac{m}{2{\kappa}},1\right\}(n+1)+\varepsilon\right)T_{f,A}(r).\label{eqN-D-0}
	\end{eqnarray}
\end{theorem}

\begin{remark}
	(a)When $m=n=\kappa$, Theorem \ref{thm1.4} recovers Theorem A. Furthermore, for $m>n$, in Theorem \ref{thm1.4}, if $m\le 2\kappa$, then the bound is $n+1$. This means that the sharp bound $n+1$ still holds so long as these divisors are in subgeneral position with index close to $m$. %Note that the sharp bound $n+1$ won't be achieved if $\frac{m}{2\kappa}$ in \eqref{eqN-D-0} is replaced by $\frac{m}{\kappa}$ due to the assumption $m\geq n\geq \kappa$.
	(b)The proof of Theorem 1.1, motivated by \cite{10-1}, is indeed very different from the proof of Theorem A which uses the Chow-Hilbert weights.  It goes back to the standard filtration method initiated by
Corvaja and Zannier (see \cite{10}). However, in \cite{10}, $X$ is required to be a complete intersection. Our extended filtration method works on arbitrary $X$. This allows us to use the joint-filtration lemma. It also gives a new proof of Theorem A.
\end{remark}
We now give an example where $m=2\kappa$.

{\noindent{\bf Example}:}
\textit{We consider the following smooth curves $D_0, D_1, D_2, D_3,D_4$ in $\mathbb{P}^2$ defined by polynomials $P_0(z)=z_0,P_1(z)=z_1, P_2(z)=z_2, P_3(z)=z_0z_1+z_2^2,P_4(z)=z_0^2z_2+z_1^3+z_1z_2^2$. Since $\bigcap_{0\leq j\leq 4} D_j= \emptyset$ and $\bigcap_{1\leq j\leq 4} D_j=\{[1,0,0]\}$, the divisor $D:=D_0+D_1+D_2+D_3+D_4$ is located in $4$-subgeneral position. In this example $m=4, \kappa=2$ and  $n=2$, the coefficient on the right hand side is $3+\varepsilon$ by Theorem \ref{thm1.4} which is known to be sharp.}

Note that, based on a similar idea, the notion of $(k,\ell)$-condition was introduced in \cite{JY} which refines the condition on divisors requiring all components to be smooth. The above example is a modification of an example given in \cite{JY}.

On the other hand, by using the notion of index, we improve Si's result (Theorem 1.1 in \cite{6.5}) as follows.

\begin{theorem}\label{thm1.3}
	Let $X$ be a complex projective variety with $\dim X=n$, and let $D_1,\ldots,D_q$ be effective Cartier divisors in $m$-subgeneral position with index $\kappa$ on $X$. Suppose that there exists an ample divisor $A$ on $X$ and positive integers $d_j$ such that $D_j\sim d_jA$ for $j=1,\ldots,q$. Let $f:{\mathbb{C}}\rightarrow X$ be an algebraically non-degenerate holomorphic curve. Then, for every $\varepsilon>0$,
	\begin{eqnarray}
	\|\sum_{j=1}^q\frac{1}{d_j}m_f(r,D_j)\le \left(\left(\frac{m-n}{\max\{1,\min\{m-n,\kappa\}\}}+1\right)(n+1)+\varepsilon\right)T_{f,A}(r).\label{eqN-0}
	\end{eqnarray}
\end{theorem}

We remark that if $m\le 2\kappa$, then (\ref{eqN-D-0}) is better than (\ref{eqN-0}). However, Theorem \ref{thm1.3} is meaningful for the case $m>2\kappa$, for example, $m=11,n=8,\kappa=2$, the coefficient on the right hand side of \eqref{eqN-0} is smaller than that of \eqref{eqN-D-0}.

Recently, Ru and Vojta \cite{7} considered the more general case, in which the divisors $D_1,\ldots,D_q$ need not be linearly equivalent. Let $\mathcal{L}$ be a line sheaf (invertible sheaf) on the projective variety $X$, and let $D$ be an effective Cartier divisor on $X$. They firstly introduced the number $\gamma(\mathcal{L},D)$ as follows. Denote by $\mathcal{L}^n$ the $n$-th tensor power $\mathcal{L}^{\otimes n}$ and denote by $\mathcal{L}(-D)$ the sheaf $\mathcal{L}\otimes\mathcal{O}(-D)$. Write $h^0(\mathcal{L}):=\dim H^0(X,\mathcal{L})$ and $h^0(\mathcal{L}(-D)):=\dim H^0(X,\mathcal{L}(-D))$. Define
$$
\gamma(\mathcal{L},D)=\limsup_{N\rightarrow \infty}\frac{Nh^0(\mathcal{L}^N)}{\sum_{\alpha=1}^{\infty}h^0(\mathcal{L}^N(-\alpha D))},
$$
where $N$ passes over all positive integers such that $h^0(\mathcal{L}^N(-D))\neq 0$. (Note that $|\mathcal{L}^N|$ is not necessarily base point free.) If $h^0(\mathcal{L}^N(-D))=0$ for all $N$, then set $\gamma(\mathcal{L},D)=+\infty$.
In \cite{7}, Ru and Vojta proved the following general theorem in terms of this number.

\noindent{\bf Theorem C.}\quad Let $X$ be a complex projective variety, and let $D_1,\ldots,D_q$ be effective Cartier divisors on $X$. Assume that $D_1,\ldots,D_q$ {\bf{intersect properly}}. i.e., for any subset $I\subset \{1,\ldots,q\}$ and any ${\bf{x}}\in\bigcap\limits_{i\in I}{\rm{Supp}}D_i$, the sequence $(\phi_i)_{i\in I}$ is a regular sequence in the local ring $\mathcal{O}_{X,{\bf{x}}}$, where $\phi_i$ is the local defining function of $D_i$. (We remark that this assumption is automatically true if $X$ is smooth and $D_1,\ldots,D_q$ are in general position on $X$.) Let $D=D_1+\cdots+D_q$ and let $\mathcal{L}$ be a line sheaf on $X$ with $h^0(\mathcal{L}^N)> 1$ for $N$ big enough. Let $f:{\mathbb{C}}\rightarrow X$ be an algebraically non-degenerate holomorphic curve. Then, for every $\varepsilon>0$,
\begin{eqnarray*}
\|m_f(r,D)\le \left(\max_{1\le j\le q}\gamma(\mathcal{L},D_j)+\varepsilon\right)T_{f,\mathcal{L}}(r).
\end{eqnarray*}

Take $\mathcal{L}=\mathcal{O}(D)$, in \cite{7}, Ru and Vojta showed that how to compute $\gamma(\mathcal{O}(D),D_j)$ for some special cases.

(i) Assume that each $D_j$ is linearly equivalent to a fixed ample divisor $A$. Then
$$
\gamma(\mathcal{O}(D),D_j)=\frac{n+1}{q}.
$$
(It means that Theorem C recovers Theorem A when $X$ is smooth.)

(ii) Assume that $D$ has equi-degree with respect to $D_1,\ldots,D_q$. i.e.,
$$
D_i\dot D^{n-1}=\frac{1}{q}D^n\ \mbox{for\ all}\ i=1,\ldots,q.
$$
(Levin \cite{8} showed that for any big and nef Cartier divisors $D_1,\ldots,D_q$, there exist positive real numbers $r_j$ such that $D=\sum\limits_{j=1}^qr_jD_j$ has equi-degree.) Then
$$
\gamma(\mathcal{O}(D),D_j)<\frac{2n}{q}.
$$

%Motivated by Ru and Vojta's result, we have the following general theorem for arbitrary effective divisors in subgeneral position.
We extend Ru and Vojta's result to arbitrary effective divisors in subgeneral position.

\begin{theorem}\label{thm1.2}
	Let $X$ be a smooth complex projective variety, and let $D_1,\ldots,D_q$ be effective divisors in $m$-subgeneral position with index ${\kappa}$ on $X$. Let $\mathcal{L}$ be a line sheaf on $X$ with $h^0(\mathcal{L}^N)> 1$ for $N$ big enough. Let $f:{\mathbb{C}}\rightarrow X$ be an algebraically non-degenerate holomorphic curve. Then, for every $\varepsilon>0$,
	\begin{eqnarray}
	\|m_f(r,D)\le \left(\frac{m}{{\kappa}}\max_{1\le j\le q}\gamma(\mathcal{L},D_j)+\varepsilon\right)T_{f,\mathcal{L}}(r),\label{eq1-4}
	\end{eqnarray}where $D=D_1+\cdots+D_q$.
\end{theorem}

\section{Preliminaries on Nevanlinna theory}\label{sec2}

In this section, we briefly recall some definitions and facts in Nevanlinna theory (cf. \cite{12}). We first introduce the definition of characteristic function.

Let $X$ be a complex projective variety and $f:{\mathbb{C}}\rightarrow X$ be a holomorphic map. Let $\mathcal{L}\rightarrow X$ be an ample line sheaf and $\omega$ be its Chern form. We define the {\bf characteristic function} of $f$ with respect to $\mathcal{L}$ by
$$
T_{f,\mathcal{L}}(r)=\int_1^r\frac{dt}{t}\int_{|z|<t}f^{*}\omega.
$$
Since any line sheaf $\mathcal{L}$ can be written as $\mathcal{L}=\mathcal{L}_1\otimes \mathcal{L}_2^{-1}$ with $\mathcal{L}_1,\ \mathcal{L}_2$ are both ample, we define $T_{f,\mathcal{L}}(r)=T_{f,\mathcal{L}_1}(r)-T_{f,\mathcal{L}_2}(r)$. A divisor $D$ on $X$ defines a line bundle $\mathcal{O}(D)$, we denote by $T_{f,D}(r)=T_{f,\mathcal{O}(D)}(r)$.
%If $\mathcal{L}$ is a line sheaf on $X$, then we define
%$$
%T_{f,\mathcal{L}}(r)=T_{f,\mathcal{O}(D)}(r)
%$$
%with $D=(s)$, $s\in H^0(X,\mathcal{L})$. It is well defined up to a bounded term.
If $X=\mathbb{P}^n(\mathbb{C})$ and $\mathcal{L}=\mathcal{O}_{\mathbb{P}^n(\mathbb{C})}(1)$, then we simply write $T_{f,\mathcal{O}_{\mathbb{P}^n(\mathbb{C})}(1)}(r)$ as $T_f(r)$.

The characteristic function satisfies the following properties:

(a) Functoriality: If $\phi:X\rightarrow X'$ is a morphism and if $\mathcal{L}$ is a line sheaf on $X'$, then
$$
T_{f,\phi^*\mathcal{L}}(r)=T_{\phi\circ f,\mathcal{L}}(r)+O(1).
$$

(b) Additivity: If $\mathcal{L}_1$ and $\mathcal{L}_2$ are line sheaves on $X$, then
$$
T_{f,\mathcal{L}_1\otimes\mathcal{L}_2}(r)=T_{f,\mathcal{L}_1}(r)+T_{f,\mathcal{L}_2}(r)+O(1).
$$

(c) Base locus: If the image of $f$ is not contained in the base locus of $|D|$, then $T_{f,D}(r)$ is bounded from below.

(d) Globally generated line sheaves: If $\mathcal{L}$ is a line sheaf on $X$, and is generated by its global sections, then $T_{f,\mathcal{L}}(r)$ is bounded from blow.

We next recall the notions of Weil functions and proximity functions.

Let $D$ be a Cartier divisor on $X$. A {\bf Weil function} with respect to $D$ is a function $\lambda_D:(X\setminus {\rm Supp}D)\rightarrow \mathbb{R}$ such that for all ${\bf x}\in X$ there is an open neighborhood $U$ of ${\bf x}$ in $X$, a nonzero rational function $f$ on $X$ with $D|_U=(f)$, and a continuous function $\alpha:U\rightarrow \mathbb{R}$ such that
$$
\lambda_D({\bf x})=-\log|f({\bf x})|+\alpha({\bf x})
$$
for all ${\bf x}\in (U\setminus {\rm Supp}D)$. Note that a continuous fiber metric $\|\cdot\|$ on the line sheaf $\mathcal{O}_{X}(D)$ determines a Weil function for $D$ given by $\lambda_D({\bf x})=-\log\|s({\bf x})\|$ where $s$ is the rational section of $\mathcal{O}_{X}(D)$ such that $D=(s)$. An example of Weil function for the hyperplane $H=\{a_0x_0+\cdots+a_nx_n=0\}$ in $\mathbb{P}^n(\mathbb{C})$ is given by
$$
\lambda_H({\bf x})=\log \frac{\max_{0\le i\le n}|x_i|\max_{0\le i\le n}|a_i|}{|a_0x_0+\cdots+a_nx_n|}.
$$

The Weil function satisfies analogues of properties which the characteristic function carries:

(a) Functoriality: If $\lambda$ is a Weil function for a Cartier divisor $D$ on $X$, and if $\phi:X'\rightarrow X$ is a morphism such that $\phi(X')\not\subset {\rm Supp}D$, then ${\bf x}\mapsto \lambda(\phi({\bf x}))$ is a Weil function for the Cartier divisor $\phi^*D$ on $X'$.

(b) Additivity: If $\lambda_1$ and $\lambda_2$ are Weil functions for Cartier divisors $D_1$ and $D_2$ on $X$, respectively, then $\lambda_1+\lambda_2$ is a Weil function for $D_1+D_2$.

(c) Uniqueness: If both $\lambda_1$ and $\lambda_2$ are Weil functions for a Cartier divisor on $X$, then $\lambda_1=\lambda_2+O(1)$.

(d) Boundedness from below: If $D$ is an effective divisor and $\lambda$ is a Weil function for $D$, then $\lambda$ is bounded from below.

(We remark that the notion of Weil function can be generalized for Weil divisors and arbitrary closed subschemes of projective varieties, see \cite{5,8-1,8-2}.)

For any holomorphic map $f:{\mathbb{C}}\rightarrow X$ whose image is not contained in the support of $D$, we define the {\bf proximity function} of $f$ with respect to $D$ by
$$
m_f(r,D)=\int_0^{2\pi}\lambda_D(f(re^{i\theta}))\frac{d\theta}{2\pi}.
$$

The proximity function satisfies the following properties:

(a) Functoriality: If $\phi:X\rightarrow X'$ is a morphism and $D'$ is a divisor on $X'$ with $\phi\circ f(\mathbb{C})\not\subset{\rm Supp}D'$, then
$$
m_f(r,\phi^*D')=m_{\phi\circ f}(r,D')+O(1).
$$

(b) Additivity: If $D_1$ and $D_2$ are two Cartier divisors on $X$, then
$$
m_f(r,D_1+D_2)=m_f(r,D_1)+m_f(r,D_2)+O(1).
$$

(c) Boundedness from below: If $D$ is effective, then $m_f(r,D)$ is bounded from below.

By the first main theorem, we have $m_f(r,D)\le T_{f,\mathcal{O}(D)}(r)+O(1)$.

To prove our main theorems, we need the following general form of Cartan's second main theorem given by Ru and Vojta \cite{7}.

\noindent{\bf Theorem D.} (Theorem 2.8 in \cite{7})\quad  Let $X$ be a complex projective variety, and let $\mathcal{L}$ be a line sheaf on $X$. Let $V$ be a linear subspace of $H^0(X,\mathcal{L})$ with $\dim V>1$, and let $s_1,\ldots,s_q$ be nonzero elements of $V$. For each $j=1,\ldots,q$, let $D_j$ be the Cartier divisor $(s_j)$. Let $f:{\mathbb{C}}\rightarrow X$ be an algebraically non-degenerate holomorphic curve. Then, for every $\varepsilon>0$,
\begin{eqnarray}
\|\int_0^{2\pi}\max_{\mathcal{K}}\sum_{j\in \mathcal{K}}\lambda_{D_j}(f(re^{i\theta}))\frac{d\theta}{2\pi}\le (\dim V+\varepsilon)T_{f,\mathcal{L}}(r),\label{eq2-1}
\end{eqnarray}
where $\max_{\mathcal{K}}$ is taken over all subsets $\mathcal{K}$ of $\{1,\ldots,q\}$ such that the sections $\{s_j\}_{j\in \mathcal{K}}$ are linearly independent.

\section{Proof of Theorems \ref{thm1.4} and \ref{thm1.3}}

Let $X$ be a complex projective variety and let $D_1,\ldots,D_q$ be effective Cartier divisors. Pick a Weil function for $D_j$, $j=1,\ldots,q$. When the divisors $D_1,\ldots,D_q$ are in $m$-subgeneral position, any point ${\bf x}\in X$ can be close to at most $m$ of the divisors $D_j$, $j=1,\ldots,q$. This implies that there exists a constant $c$ such that, for any ${\bf x}\in X\setminus \bigcup\limits_{j=1}^q{\rm Supp} D_j$, there are indices $j_1,\ldots,j_m\subset\{1,\ldots,q\}$ such that
\begin{eqnarray}
\sum_{j=1}^q\lambda_{D_j}({\bf x})\le \sum_{i=1}^m\lambda_{D_{j_i}}({\bf x})+c.\label{eq3-1}
\end{eqnarray}
(The proof of \eqref{eq3-1} is similar to that of Lemma 20.7 in \cite{12} which is standard and is omitted here.)

%\begin{defi}[cf. \cite{7}, Definition 4.11]
%	Let $V$ be a vector space of finite dimension. A \textbf{filtration} of $V$ is a family of subspaces $\mathcal{F}=(\mathcal{F}_x)_{x\in\mathbb{R}_\geq 0}$ of subspaces of $V$ such that $\mathcal{F}_x\supset\mathcal{F}_y$ whenever $x\leq y$, and such that $\mathcal{F}_x=\{0\}$ for $x$ big enough. A basis $\mathcal{B}$ of $V$ is said to be \textbf{with respect to} $\mathcal{F}$ if $\mathcal{B}\cap\mathcal{F}_x$ is a basis of $\mathcal{F}_x$ for every real number $x\geq 0$.
%\end{defi}

We need the following linear algebra lemma.

\begin{lemma}[Lemma 10.1 in \cite{8}]\label{lem3.1}
	Let $V$ be a vector space of finite dimension $\ell$. Let $V=W_1\supset W_2\supset\cdots \supset W_h$ and $V={W}'_1\supset {W}'_2\supset\cdots \supset {W}'_{h'}$ be two filtrations of $V$. There exists a basis $v_1,\ldots,v_\ell$ of $V$ which contains a basis of each $W_j$ and ${W}'_j$.
\end{lemma}
%\begin{lemma}[cf. \cite{8}, Lemma 10.1]\label{lem3.1}
%	Let $\mathcal{F}$ and $\mathcal{G}$ be two filtrations of $V$, then there is a basis of $W$ which is with respect to both $\mathcal{F}$ and $\mathcal{G}$.
%\end{lemma}

Let $A$ be an ample divisor and let $d_1\ldots,d_q$ be positive integers such that $D_j\sim d_jA$ for $j=1,\ldots,q$.
Set $d$ to be the least common multiple of $d_j$'s. Since $D_j\sim d_jA$, we have $\frac{d}{d_j}D_j\sim dA$ and
$$
m_f(r,D_j)=\frac{d_j}{d}m_f(r,(d/d_j)D_j)+O(1)\ \ \ \ \ \mbox{for}\ j=1,\ldots,q.
$$
Replacing $D_j$ by $\frac{d}{d_j}D_j$,
it suffices to prove
$$
\|\sum_{j=1}^q\frac{1}{d}m_f(r,(d/d_j)D_j)\le \left(\max\left\{\frac{m}{2{\kappa}},1\right\}(n+1)+\varepsilon\right)T_{f,A}(r)
$$
and
$$
\|\sum_{j=1}^q\frac{1}{d}m_f(r,(d/d_j)D_j)\le \left(\left(\frac{m-n}{\max\{1,\min\{m-n,\kappa\}\}}+1\right)(n+1)+\varepsilon\right)T_{f,A}(r).
$$
Hence, we assume that $D_j\sim dA$ in the rest of the proofs of Theorem \ref{thm1.4} and Theorem \ref{thm1.3}.

Let $\widetilde{N}$ be a positive integer such that $\widetilde{N}A$ is very ample and $\widetilde{N}$ is divisible by $dq$. Let $\phi:X\rightarrow \mathbb{P}^M(\mathbb{C})$ be the canonical embedding of $X$ into $\mathbb{P}^M(\mathbb{C})$ associated to $\widetilde{N}A$, where $M=h^0(\mathcal{O}(\widetilde{N}A))-1$. Then since $\frac{\widetilde{N}}{d}D_j\sim \widetilde{N}A$, $\frac{\widetilde{N}}{d}D_j=\phi^*H_j$ for some hyperplane $H_j$ in $\mathbb{P}^M(\mathbb{C})$, $j=1,\ldots,q$. Denote by $L_j$ the defining linear form of $H_j$ (i.e., $H_j=\{L_j=0\}$) for $j=1,\ldots,q$. By the assumption that $D_1,\ldots,D_q$ are in $m$-subgeneral position with index $\kappa$, $H_1,\ldots,H_q$ are in $m$-subgeneral position with index $\kappa$ on $\phi(X)\subset \mathbb{P}^M(\mathbb{C})$. From the functoriality and additivity of Weil functions, we have
\begin{eqnarray}
\lambda_{H_j}(\phi({\bf{x}}))=\frac{\widetilde{N}}{d}\lambda_{D_j}({\bf{x}})+O(1)\ \ \
{\rm for}\ \ \
{\bf{x}}\in X\setminus{\rm{Supp}}D_j.\label{eqN-1-a}
\end{eqnarray}
Hence,
\begin{eqnarray}
m_{\phi\circ f}(r,H_j)=\frac{\widetilde{N}}{d}m_f(r,D_j)+O(1).\label{eqN-1}
\end{eqnarray}

\begin{proof}[{\bf Proof of Theorem \ref{thm1.4}}]

Consider $\mathcal{L}=\mathcal{O}(D)$ with $D=D_1+\cdots+D_q$. Choose a positive integer $N$ and let $\widehat{N}=N\widetilde{N}/dq$. Note that
$$
H^0(X,\mathcal{L}^{\widehat{N}})=H^0(X,\mathcal{L}^{\frac{N\widetilde{N}}{dq}})\supset \phi^*H^0(\phi(X),\mathcal{O}_{\phi(X)}(N)).
$$
By (19) of \cite{10-2}, for $N$ big enough, we have
$$
H^0(\phi(X),\mathcal{O}_{\phi(X)}(N))\cong\frac{\mathbb{C}[x_0,\ldots,x_M]_N}{I(\phi(X))_N},
$$
where $I(\phi(X))$ is the ideal of $\mathbb{C}[x_0,\ldots,x_M]$ with respect to $\phi(X)$ and $I(\phi(X))_N=I(\phi(X))\cap \mathbb{C}[x_0,\ldots,x_M]_N$.

Let $V_N=\mathbb{C}[x_0,\ldots,x_M]_N$ and $\widehat{V_N}=\frac{\mathbb{C}[x_0,\ldots,x_M]_N}{I(\phi(X))_N}$. For $g\in V_N$, denote by $[g]$ the projection of $g$ in $\widehat{V_N}$. Let $H_{\phi(X)}(N)$ be the Hilbert function of $\phi(X)$. We have
\begin{eqnarray}
\dim \widehat{V_N}=H_{\phi(X)}(N)=\deg\phi(X)\frac{N^n}{n!}+O(N^{n-1})\label{eq3-15}
\end{eqnarray}
for $N$ big enough.

Suppose first that $X$ is normal. Given $z\in \mathbb{C}$, we arrange so that
$$
\lambda_{D_{1,z}}(f(z))\ge \lambda_{D_{2,z}}(f(z))\ge\cdots\ge \lambda_{D_{{\kappa},z}}(f(z))\ge\cdots\ge\lambda_{D_{m,z}}(f(z))\ge\cdots\ge\lambda_{D_{q,z}}(f(z)),
$$
then, by the argument of (\ref{eq3-1}) and (\ref{eqN-1-a}),
\begin{eqnarray}
\frac{\widetilde{N}}{d}\sum_{j=1}^q\lambda_{D_j}(f(z))\le \sum_{j=1}^m\lambda_{H_{j,z}}(\phi\circ f(z))+O(1).\label{eq3-16}
\end{eqnarray}

By the assumption that $H_1,\ldots,H_q$ are in $m$-subgeneral position with index ${\kappa}$ on $\phi(X)$, we have ${\rm codim}\bigcap\limits_{j=1}^{\kappa}H_{j,z}\cap\phi(X)={\kappa}$. We may take hyperplanes $H'_{{\kappa}+1,z},\ldots,H'_{n,z}$ such that
$$
H_{1,z},\ldots,H_{{\kappa},z},H'_{{\kappa}+1,z},\ldots,H'_{n,z}
$$
are in general position on $\phi(X)$.

Set $$\{r_1,\ldots,r_n\}:=\{L_{1,z},\ldots,L_{{\kappa},z},L'_{{\kappa}+1,z},\ldots,L'_{n,z}\},$$ where $L_{1,z},\ldots,L_{{\kappa},z},L'_{{\kappa}+1,z},\ldots,L'_{n,z}$ are the defining linear forms of above hyperplanes. We now introduce a filtration of $\widehat{V_{N}}$ with respect to $\{r_1,\ldots,r_n\}$, which is a generalization of Corvaja-Zannier's filtration (in \cite{10,1}) given by Cherry, Dethloff and Tran \cite{10-1}. We remark that, in \cite{10-1}, Cherry, Dethloff and Tran constructed such filtration to deal with the moving targets problem. We will focus on the fundamental case of fixed targets. We also make improvements by using
joint filtration method which is helpful to achieve the sharp bound.%But, it is more useful in our case and we can make improvement by using joint filtration method.%, in which the regular sequence condition can be dropped.
%so that the filtration and joint filtration method can be used without regular sequence condition.

Arrange, by the lexicographic order, the $n$-tuples ${\bf i}=(i_1,\ldots,i_n)$ of
non-negative integers and set $\|{\bf i}\|=\sum_ji_j$.

\begin{defi}
	(i) For each ${\bf i}\in \mathbb{Z}_{\ge 0}^n$ and non-negative integer $N$ with $N\ge \|{\bf i}\|$, denote by $I_N^{\bf i}$ the subspace of $\mathbb{C}[x_0,\ldots,x_M]_{N-\|{\bf i}\|}$ consisting of all $r\in\mathbb{C}[x_0,\ldots,x_M]_{N-\|{\bf i}\|}$ such that
$$
r_1^{i_1}\cdots r_n^{i_n}r-\sum_{{\bf e}=(e_1,\ldots,e_n)>{\bf i}}r_1^{e_1}\cdots r_n^{e_n}r_{\bf e}\in I(\phi(X))_N
$$
(or
$[r_1^{i_1}\cdots r_n^{i_n}r]=[\sum_{{\bf e}=(e_1,\ldots,e_n)>{\bf i}}r_1^{e_1}\cdots r_n^{e_n}r_{\bf e}]$ on $\widehat{V_{N}}$)
for some $r_{\bf e}\in \mathbb{C}[x_0,\ldots,x_M]_{N-\|{\bf e}\|}$.

(ii) Denote by $I^{\bf i}$ the homogeneous ideal in $\mathbb{C}[x_0,\ldots,x_M]$ generated by $\bigcup_{N\ge\|{\bf i}\|}I^{\bf i}_N$.
\end{defi}

\begin{remark}\label{rkN3.3}
From this definition, we have the following properties.

(i) $(I(\phi(X)),r_1,\ldots,r_n)_{N-\|{\bf i}\|}\subset I^{\bf i}_N\subset\mathbb{C}[x_0,\ldots,x_M]_{N-\|{\bf i}\|}$, where $(I(\phi(X)),r_1,\ldots,r_n)$ is the ideal in $\mathbb{C}[x_0,\ldots,x_M]$ generated by $I(\phi(X))\cup\{r_1,\ldots,r_n\}$.

(ii) $I^{\bf i}\cap\mathbb{C}[x_0,\ldots,x_M]_{N-\|{\bf i}\|}=I^{\bf i}_N$.

%(iii) $\frac{\mathbb{C}[x_0,\ldots,x_M]}{I^{\bf i}}$ is a graded module over $\mathbb{C}[x_0,\ldots,x_M]$.

(iii) If ${\bf i}_1-{\bf i}_2:=(i_{1,1}-i_{2,1},\ldots,i_{1,n}-i_{2,n})\in \mathbb{Z}_{\ge 0}^n$, then $I_{N}^{{\bf i}_2}\subset I_{N+\|{\bf i}_1\|-\|{\bf i}_2\|}^{{\bf i}_1}$. Hence $I^{{\bf i}_2}\subset I^{{\bf i}_1}$.
\end{remark}

\begin{lemma}\label{lemN3.4}
	$\{I^{\bf i}|{\bf i}\in \mathbb{Z}_{\ge 0}^n\}$ is a finite set.
\end{lemma}
\begin{proof}
Suppose that $\sharp\{I^{\bf i}|{\bf i}\in \mathbb{Z}_{\ge 0}^n\}=\infty$. We can construct a sequence $\{{\bf i}_t\}_{t=1}^{\infty}$ such that ${\bf i}_{t+1}-{\bf i}_{t}\in \mathbb{Z}_{\ge 0}^n$ and $\{I^{{\bf i}_t}\}_{t=1}^{\infty}$ consisting of pairwise different ideals. By (iii) of Remark \ref{rkN3.3},
$$
I^{{\bf i}_{1}}\subset I^{{\bf i}_{2}}\subset \cdots \subset I^{{\bf i}_{t}}\subset I^{{\bf i}_{t+1}}\subset\cdots,
$$
which contradicts the fact that $\mathbb{C}[x_0,\ldots,x_M]$ is a Noetherian ring.
\end{proof}

Denote
\begin{eqnarray}
\Delta_N^{\bf i}:=\dim\frac{\mathbb{C}[x_0,\ldots,x_M]_{N-\|{\bf i}\|}}{I^{\bf i}_{N}}.\label{eqN-D-1}
\end{eqnarray}

\begin{lemma}\label{lemN3.5}
	(i) There exists a positive integer $N_0$ such that, for each ${\bf i}\in \mathbb{Z}^n_{\ge 0}$, $\Delta_N^{\bf i}$ is independent of $N$ for all $N$ satisfying $N-\|{\bf i}\|>N_0$.

(ii) For all $N$ and ${\bf i}$ with $N-\|{\bf i}\|\ge 0$, $\Delta_N^{\bf i}$ is bounded.
\end{lemma}
\begin{proof}
%(i) We note that $H_{1,z},\ldots,H_{{\kappa},z},H'_{{\kappa}+1,z},\ldots,H'_{n,z}$
%are in general position on $\phi(X)$, i.e.,
%$$
%\dim H_{1,z}\cap\ldots\cap H_{{\kappa},z}\cap H'_{{\kappa}+1,z}\cap\ldots\cap H'_{n,z}\cap\phi(X)=0.
%$$
%By the theory of Hilbert function, there is an integer $N_1$ such that
%$$
%\dim\frac{\mathbb{C}[x_0,\ldots,x_M]_{N-\|{\bf i}\|}}{(I(\phi(X)),r_1,\ldots,r_n)_{N-\|{\bf i}\|}}
%$$
%is a constant for all ${\bf i}\in \mathbb{Z}_{\ge 0}^{n}$ and $N$ with $N-\|{\bf i}\|>N_1$ (See \cite{Hart}, Theorem 7.5).
%
%For each ${\bf i}\in \mathbb{Z}_{\ge 0}^{n}$, by (iii) of Remark \ref{rkN3.3},
%$$
%\Delta_N^{\bf i}=\dim\frac{\mathbb{C}[x_0,\ldots,x_M]_{N-\|{\bf i}\|}}{I^{\bf i}_{N}}
%$$
%is a polynomial of $N$ for $N$ big enough (See \cite{Hart}, Theorem 7.5). By (i) of Remark \ref{rkN3.3},
%\begin{eqnarray}
%\dim\frac{\mathbb{C}[x_0,\ldots,x_M]_{N-\|{\bf i}\|}}{I^{\bf i}_{N}}\le\dim\frac{\mathbb{C}[x_0,\ldots,x_M]_{N-\|{\bf i}\|}}{(I(\phi(X)),r_1,\ldots,r_n)_{N-\|{\bf i}\|}}.\label{eqN-D-2}
%\end{eqnarray}
%Hence, there is an integer $N_2^{\bf i}(>N_1)$ such that $\Delta_N^{\bf i}$ is also a constant for $N-\|{\bf i}\|>N_2^{\bf i}$. Set $\Delta^{\bf i}$ to be this constant. We note $N_2^{\bf i}$ depends on $I^{\bf i}$ and $\{I^{\bf i}|{\bf i}\in \mathbb{Z}_{\ge 0}^n\}$ is a finite set by Lemma \ref{lemN3.4}. Take $N_0=\max\{N_2^{\bf i}|{\bf i}\in \mathbb{Z}_{\ge 0}^n\}$, we have $\Delta^{\bf i}=\Delta^{\bf i}_N$ for all ${\bf i}\in \mathbb{Z}_{\ge 0}^n$ and $N-\|{\bf i}\|>N_0$.

For each ${\bf i}\in \mathbb{Z}_{\ge 0}^{n}$, $\frac{\mathbb{C}[x_0,\ldots,x_M]}{I^{\bf i}}$ is a graded $\mathbb{C}[x_0,\ldots,x_M]$-module. By (i) of Remark \ref{rkN3.3}, we have $(I(\phi(X)),r_1,\ldots,r_n)\subset{\rm Ann}\left(\frac{\mathbb{C}[x_0,\ldots,x_M]}{I^{\bf i}}\right)$. Since $r_1,\ldots,r_n$
are located in general position on $\phi(X)$, $\dim Z\left({\rm Ann}\left(\frac{\mathbb{C}[x_0,\ldots,x_M]}{I^{\bf i}}\right)\right)\le 0$ where $Z(\cdot)$ denotes the zero set in $\mathbb{P}^M(\mathbb{C})$ of a homogeneous ideal. By the Hilbert-Serre Theorem (see \cite{Hart}), there exist non-negative integers $N(I^{\bf i})$ and $\Delta(I^{\bf i})$ such that $\Delta_N^{\bf i}=\Delta(I^{\bf i})$ when $N-\|{\bf i}\|>N(I^{\bf i})$.  Set $N_0=\max\{N(I^{\bf i})|{\bf i}\in \mathbb{Z}_{\ge 0}^n\}$ and $\widetilde{\Delta} = \max\{\Delta(I^{\bf i})|{\bf i}\in \mathbb{Z}_{\ge 0}^n\}$ ($\{I^{\bf i}|{\bf i}\in \mathbb{Z}_{\ge 0}^n\}$ is a finite set by Lemma \ref{lemN3.4}), then we have $\Delta^{\bf i}_N=\Delta(I^{\bf i})\leq \widetilde{\Delta}$ when $N-\|{\bf i}\|>N_0$ which gives the conclusion (i). To see (ii), it suffices to combine the above inequality with $\Delta^{\bf i}_N\leq \dim\mathbb{C}[x_0,\ldots,x_M]_{N-\|{\bf i}\|}={N-\|{\bf i}\|+M\choose M}$.
\end{proof}

\begin{remark}\label{rkN3.6}For simplicity, we will rewrite the integer $\Delta(I^{\bf i})$ in the proof of Lemma \ref{lemN3.5} as $\Delta^{\bf i}$ for each ${\bf i}\in\mathbb{Z}_{\ge 0}^n$. Set $\Delta_0:=\min_{{\bf i}\in \mathbb{Z}_{\ge 0}^n}\Delta^{\bf i},$ then
$\Delta_0=\Delta^{{\bf i}_0}$ for some ${\bf i}_0\in \mathbb{Z}_{\ge 0}^n$. By (iii) of Remark \ref{rkN3.3}, if ${\bf i}-{\bf i}_0\in \mathbb{Z}_{\ge 0}^n$, then $\Delta^{\bf i}\le \Delta^{{\bf i}_0}$.
\end{remark}

Now, for $N$ big enough with $N>N_0$ and $N-\|{\bf i}_0\|>0$, we construct the following filtration of $\widehat{V_N}$ with respect to $\{r_1,\ldots,r_n\}$.

Denote by $\tau_N$ the set of ${\bf i}\in \mathbb{Z}_{\ge 0}^n$ with $N-\|{\bf i}\|\ge 0$, arranged by
the lexicographic order.

Define
the spaces $W_{{\bf i}}=W_{N,{\bf i}}$ by
$$
W_{{\bf i}}=\sum_{{\bf e}\ge{\bf i}}r_1^{e_1}\cdots r_n^{e_n}V_{N-\|{\bf e}\|}.
$$
Plainly $W_{(0,\ldots,0)}=V_N$ and $W_{{\bf i}}\supset W_{{\bf i}'}$ if
${\bf i}'> {\bf i}$, so the $W_{{\bf i}}$ is a filtration of $V_N$. Set $\widehat{W_{{\bf i}}}=\{[g]|g\in W_{{\bf i}}\}$. Hence, the $\widehat{W_{{\bf i}}}$ is a filtration of $\widehat{V_N}$.

\begin{lemma}\label{lemN3.7}
	Suppose that ${\bf i}'$ follows ${\bf i}$ in the lexicographic order, then
$$
\frac{\widehat{W_{\bf i}}}{\widehat{W_{{\bf i}'}}}\cong\frac{\mathbb{C}[x_0,\ldots,x_M]_{N-\|{\bf i}\|}}{I^{\bf i}_{N}}.
$$
\end{lemma}
\begin{proof}
Define a vector space homomorphism
$$
\varphi:\mathbb{C}[x_0,\ldots,x_M]_{N-\|{\bf i}\|}\rightarrow \frac{\widehat{W_{\bf i}}}{\widehat{W_{{\bf i}'}}},
$$
which maps $r\in \mathbb{C}[x_0,\ldots,x_M]_{N-\|{\bf i}\|}$ to $[r_1^{i_1}\cdots r_n^{i_n}r](\in \widehat{W_{\bf i}})$ modulo $\widehat{W_{{\bf i}'}}$. Obviously, it is surjective.

Let $\ker \varphi$ be the kernel of $\varphi$. Suppose $r\in \ker \varphi$. This means
$$
[r_1^{i_1}\cdots r_n^{i_n}r]\in \widehat{W_{{\bf i}'}}
$$
(or
$[r_1^{i_1}\cdots r_n^{i_n}r]=[\sum_{{\bf e}=(e_1,\ldots,e_n)>{\bf i}}r_1^{e_1}\cdots r_n^{e_n}r_{\bf e}]$)
for some $r_{\bf e}\in V_{N-\|{\bf e}\|}$. i.e., $r\in I_N^{\bf i}$. Hence, $\ker \varphi\subset I_N^{\bf i}$. On the other hand, if $r\in I_N^{\bf i}$, then, there exist $r_{\bf e}\in V_{N-\|{\bf e}\|}$ such that
$$
[r_1^{i_1}\cdots r_n^{i_n}r]=[\sum_{{\bf e}=(e_1,\ldots,e_n)>{\bf i}}r_1^{e_1}\cdots r_n^{e_n}r_{\bf e}]\in \widehat{W_{\bf i}}.
$$
i.e., $r\in \ker \varphi$. Hence, $\ker \varphi= I_N^{\bf i}$, which concludes the proof of Lemma \ref{lemN3.7}.
\end{proof}

Combining with (\ref{eqN-D-1}), we have
$$
\dim\frac{\widehat{W_{\bf i}}}{\widehat{W_{{\bf i}'}}}=\Delta^{{\bf i}}_N.
$$

Set
$$
\tau^0_N=\{{\bf i}\in \tau_N|N-\|{\bf i}\|>N_0\ \mbox{and}\ {\bf i}-{\bf i}_0\in\mathbb{Z}_{\ge 0}^n\}.
$$

\begin{lemma}\label{lemN3.8}
	(i) $\Delta_0=\Delta^{\bf i}$ for all ${\bf i}\in \tau^0_N$.

(ii) $\sharp\tau^0_N=\frac{N^{n}}{n!}+O(N^{n-1})$.

(iii) $\Delta^{\bf i}_N=\deg \phi(X)$ for all ${\bf i}\in \tau^0_N$.
\end{lemma}
\begin{proof}
(i) By the definition of $\tau^0_N$, we have $\Delta^{\bf i}_N=\Delta^{\bf i}$ for ${\bf i}\in \tau^0_N$. On the other hand, $\Delta^{\bf i}\le \Delta^{{\bf i}_0}$ (note that ${\bf i}-{\bf i}_0\in \mathbb{Z}_{\ge 0}^n$ and Remark \ref{rkN3.6}). By the minimality of $\Delta^{{\bf i}_0}$, we have (i).

(ii) We have
$$
\sharp\tau_N=\left(\begin{matrix}N+n\\n\end{matrix}\right)=\frac{N^{n}}{n!}+O(N^{n-1}),\ \ \
\sharp\{{\bf i}\in\tau_N|N-\|{\bf i}\|\le N_0\}=O(N^{n-1})
$$
and
$$
\sharp\{{\bf i}\in\tau_N|{\bf i}-{\bf i}_0=(i_1-i_{0,1},\ldots,i_n-i_{0,n})\ \mbox{with\ some}\ i_l-i_{0,l}<0\}=O(N^{n-1}).
$$
It implies that $\sharp\tau^0_N=\frac{N^{n}}{n!}+O(N^{n-1})$.

(iii) By (ii) of Lemma \ref{lemN3.5}, $\Delta_N^{{\bf i}}$ is bounded by all ${\bf i}$ and $N$. Hence, combining (i), (ii) and (\ref{eq3-15}),
\begin{eqnarray*}
\deg\phi(X)\frac{N^n}{n!}+O(N^{n-1})&=&\sum_{{\bf i}\in \tau_N}\Delta_N^{\bf i}=\Delta_0\cdot \sharp\tau^0_N+\sum_{{\bf i}\in \tau_N\setminus\tau^0_N}\Delta_N^{\bf i}\\
&=&\Delta_0\left(\frac{N^n}{n!}+O(N^{n-1})\right)+O(N^{n-1}).
\end{eqnarray*}
We have $\Delta_0=\deg\phi(X)$.
\end{proof}

We choose a basis $\mathcal{B}=\{s_1,\ldots,s_{H_{\phi(X)}(N)}\}$ of $\widehat{V_N}$ with respect to the above filtration. Let $s$ be an element of the basis, which lies in
$\widehat{W}_{{\bf i}}\setminus \widehat{W}_{{\bf i}'}$, we may write $s=[r_1^{i_1}\cdots
r_n^{i_n}r]$, where $r\in V_{N-\|{\bf i}\|}$. For every $1\le j\le n$, we have
\begin{eqnarray}
\sum_{{\bf i}\in \tau_N}\Delta_N^{\bf i}i_j=\deg\phi(X) \frac{N^{n+1}}{(n+1)!}+O(N^{n}).\label{eq3-17}
\end{eqnarray}
(The proof of (\ref{eq3-17}) is similar to (3.6) in \cite{1}.) Hence, for any prime divisor $E$,
$$
\sum_{i=1}^{H_{\phi(X)}(N)}{\rm ord}_{E}\phi^*(s_i)\ge \left(\deg\phi(X) \frac{N^{n+1}}{(n+1)!}+O(N^{n})\right){\rm ord}_{E}(\phi^*H_{1,z}+\cdots+\phi^*H_{{\kappa},z}),
$$
where $\phi^*(s_i)$ is the divisor on $X$ with respect to $s_i$, i.e.,
$$
\sum_{i=1}^{H_{\phi(X)}(N)}\phi^*(s_i)\ge \left(\deg\phi(X) \frac{N^{n+1}}{(n+1)!}+O(N^{n})\right)\cdot \frac{\widetilde{N}}{d}\cdot (D_{1,z}+\cdots+D_{{\kappa},z}).
$$

We may construct another filtration of $\widehat{V_N}$ by using $H_{{\kappa}+1,z},\ldots,H_{2{\kappa},z}$ if $2{\kappa}\le m$, otherwise $H_{{\kappa}+1,z},\ldots,H_{m,z}$.

By Lemma \ref{lem3.1}, we can construct a basis $\mathcal{B}=\{s_1,\ldots,s_{H_{\phi(X)}(N)}\}$ with respect to the above two filtrations such that
\begin{eqnarray}
\sum_{i=1}^{H_{\phi(X)}(N)}\phi^*(s_i)\ge \left(\deg\phi(X) \frac{N^{n+1}}{(n+1)!}+O(N^{n})\right)\cdot \frac{\widetilde{N}}{d}\cdot(D_{1,z}+\cdots+D_{{\kappa},z})\label{eq3-18}
\end{eqnarray}
and
\begin{eqnarray}
\sum_{i=1}^{H_{\phi(X)}(N)}\phi^*(s_i)\ge \left(\deg\phi(X) \frac{N^{n+1}}{(n+1)!}+O(N^{n})\right)\cdot \frac{\widetilde{N}}{d}\cdot(D_{{\kappa}+1,z}+\cdots+D_{2{\kappa},z}).\label{eq3-19}
\end{eqnarray}
(If $2{\kappa}>m$, $D_{{\kappa}+1,z}+\cdots+D_{2{\kappa},z}$ in (\ref{eq3-19}) should be replaced by $D_{{\kappa}+1,z}+\cdots+D_{m,z}$.)

%Since $X$ is normal, we may view Cartier divisors as Weil divisors. For two effective divisors $D,E$ on $X$, define
%$$
%{\rm gcd}(D,E)=\sum\min\{{\rm ord}_FD,{\rm ord}_FE\}F
%$$
%and
%$$
%{\rm lcm}(D,E)=\sum\max\{{\rm ord}_FD,{\rm ord}_FE\}F,
%$$
%where the sums run over all prime divisors $F$. We call $D$ and $E$ have no common components if ${\rm gcd}(D,E)=0$ (or ${\rm lcm}(D,E)=D+E$).

We remark that if $D_1,\ldots,D_q$ are in $m$-subgeneral position with index ${\kappa}>1$, then $D_i$ and $D_j$ have no common components for $i\neq j$, combining \eqref{eq3-18} and \eqref{eq3-19}, we have
\begin{equation}
\sum_{i=1}^{H_{\phi(X)}(N)}\phi^*(s_i)\geq
%&\ge&\left(\deg\phi(X) \frac{N^{n+1}}{(n+1)!}+O(N^{n})\right)\cdot \frac{\widetilde{N}}{d}\cdot{\rm lcm}\left(\sum_{j=1}^{\kappa}D_{j,z},\sum_{j=\kappa+1}^{2\kappa}D_{j,z}\right)\nonumber\\
% &=&
 \left(\deg\phi(X) \frac{N^{n+1}}{(n+1)!}+O(N^{n})\right)\cdot \frac{\widetilde{N}}{d}\cdot(D_{1,z}+\cdots+D_{2{\kappa},z}).\label{eq3-20}
\end{equation}
(If $2{\kappa}>m$, $D_{1,z}+\cdots+D_{2{\kappa},z}$ in (\ref{eq3-20}) should be replaced by $D_{1,z}+\cdots+D_{m,z}$.) By (\ref{eq3-16}) and (\ref{eq3-20}),
\begin{eqnarray}
\frac{\widetilde{N}}{d}\sum_{j=1}^q\lambda_{D_j}(f(z))&\le& \max\left\{\frac{m}{2{\kappa}},1\right\}\frac{1}{\deg\phi(X) \frac{N^{n+1}}{(n+1)!}+O(N^{n})}\sum_{i=1}^{H_{\phi(X)}(N)}\lambda_{\phi^*(s_i)}(f(z))\nonumber\\
&&+O(1).\label{eq3-21}
\end{eqnarray}
As there are only finitely many choices of
$\{r_1,\ldots,r_n\}$, we have a finite
collection of bases $\mathcal{B}$. By Theorem D with $\varepsilon=\frac{1}{2}$, (\ref{eq3-21}) implies that
\begin{eqnarray}
&&\|\widetilde{N}\sum_{j=1}^q\frac{1}{d}m_f(r,D_j)\nonumber\\
&\le& \max\left\{\frac{m}{2{\kappa}},1\right\}\frac{1}{\deg\phi(X) \frac{N^{n+1}}{(n+1)!}+O(N^{n})}\int_0^{2\pi}\max_{\mathcal{B}}\sum_{s_{j}\in \mathcal{B}}\lambda_{\phi^*(s_{j})}(f(re^{i\theta}))\frac{d\theta}{2\pi}+O(1)\nonumber\\
&\le&\max\left\{\frac{m}{2{\kappa}},1\right\}\frac{1}{\deg\phi(X) \frac{N^{n+1}}{(n+1)!}+O(N^{n})}(H_{\phi(X)}(N)+1/2)T_{f,\mathcal{L}^{\widehat{N}}}(r)+O(1),\label{eq3-22}
\end{eqnarray}
where $\max_{\mathcal{B}}$ is taken over all bases constructed above.

Note that $T_{f,\mathcal{L}^{\widehat{N}}}(r)=NT_{f,\mathcal{O}(\widetilde{N}A)}(r)+O(1)$. From (\ref{eq3-15}) and (\ref{eq3-22}), it follows that for $N$ big enough,
\begin{eqnarray*}
&&\|\widetilde{N}\sum_{j=1}^q\frac{1}{d}m_f(r,D_j)\\
&\le& \max\left\{\frac{m}{2{\kappa}},1\right\}\frac{1}{\deg\phi(X) \frac{N^{n+1}}{(n+1)!}+O(N^{n})}\left(\deg\phi(X) \frac{N^{n}}{n!}+O(N^{n-1})+\frac{1}{2}\right)N\cdot T_{f,\widetilde{N}A}(r)\\
&&+O(1)\\
&\le&\max\left\{\frac{m}{2{\kappa}},1\right\}(n+1+\varepsilon)\widetilde{N}T_{f,A}(r).
\end{eqnarray*}
Hence, (\ref{eqN-D-0}) holds for this case.

If $X$ is not normal, then we consider the normalization $\pi:\widetilde{X}\rightarrow X$ and divisors $\pi^*A$ and $\pi^*D_j$. We have $\pi^*A$ is ample and $\pi^*D_1,\ldots,\pi^*D_q$ are in $m$-subgeneral position with index $\kappa$ on $\widetilde{X}$. Hence
$$
\|\sum_{j=1}^q\frac{1}{d}m_{\widetilde{f}}(r,D_j)\le \left(\max\left\{\frac{m}{2{\kappa}},1\right\}(n+1)+\varepsilon\right)T_{\widetilde{f},\pi^*A}(r),
$$
where $\widetilde{f}:\mathbb{C}\rightarrow \widetilde{X}$ is the lifting of $f$. As $\pi$ is a birational morphism, by the properties of Weil function and characteristic function, we have
(\ref{eqN-D-0}) for the general case.
\end{proof}

%In particular, motivated by the proof of Theorem \ref{thm1.4}, we can obtain the following result in which we do not make the assumption that those divisors are in general position or subgeneral position. Actually, we give a quantitative version of a well-known result that every holomorphic curve $f:\mathbb{C}\rightarrow \mathbb{P}^2\setminus D$ is algebraically degenerate if $D$ has 4 distinct components.
%
%\begin{theorem}\label{coro1.6}
%	Let $D_j$, $j=1,2,3,4$, be the distinct irreducible hypersurface of degree $d_j$ in $\mathbb{P}^2(\mathbb{C})$. Let $f:{\mathbb{C}}\rightarrow \mathbb{P}^2(\mathbb{C})$ be an algebraically non-degenerate holomorphic curve. Then, for every $\varepsilon>0$,
%	\begin{eqnarray*}
%		\|\sum_{j=1}^4\frac{1}{d_j}m_f(r,D_j)\le (3+\varepsilon)T_f(r).
%	\end{eqnarray*}
%\end{theorem}
%
%\begin{remark}\label{rk3.3}
%	 We now explain how to obtain Theorem \ref{coro1.6}. For $z\in \mathbb{C}$, assume that
%	 $$
%	 \lambda_{D_{1,z}}(f(z))\ge \lambda_{D_{2,z}}(f(z))\ge \lambda_{D_{3,z}}(f(z))\ge\lambda_{D_{4(z)}}(f(z)).
%	 $$
%	 Since $D_i$ and $D_j$ have no common components for $i\neq j$, $D_{1,z},D_{2,z}$ are in general position and $D_{3,z},D_{4(z)}$ are also in general position. By the method used in the proof of Theorem \ref{thm1.4}, one can easily show Theorem \ref{coro1.6}.
%\end{remark}

Next, we give a quantitative version of a well-known result that every holomorphic curve $f:\mathbb{C}\rightarrow \mathbb{P}^2\setminus D$ is algebraically degenerate if $D$ has 4 distinct components.

\begin{coro}\label{coro1.6}
	Let $D_j$, $j=1,2,3,4$, be the distinct irreducible hypersurfaces of degree $d_j$ in $\mathbb{P}^2(\mathbb{C})$. Let $f:{\mathbb{C}}\rightarrow \mathbb{P}^2(\mathbb{C})$ be an algebraically non-degenerate holomorphic curve. Then, for every $\varepsilon>0$,
	\begin{eqnarray*}
		\|\sum_{j=1}^4\frac{1}{d_j}m_f(r,D_j)\le (3+\varepsilon)T_f(r).
	\end{eqnarray*}
\end{coro}

%\begin{remark}\label{rk3.3}
%	In fact, for $z\in \mathbb{C}$, assume that
%	$$
%	\lambda_{D_{1,z}}(f(z))\ge \lambda_{D_{2,z}}(f(z))\ge \lambda_{D_{3,z}}(f(z))\ge\lambda_{D_{4,z}}(f(z)).
%	$$
%	Since $D_i$ and $D_j$ have no common components for $i\neq j$, $D_{1,z},D_{2,z}$ are in general position and $D_{3,z},D_{4,z}$ are also in general position. The method used in the proof of Theorem \ref{thm1.4} then leads to Theorem \ref{coro1.6}.
%\end{remark}
Corollary \ref{coro1.6} follows from Theorem \ref{thm1.4} with $m\leq 4,n=\kappa=2$.

\begin{proof}[{\bf Proof of Theorem \ref{thm1.3}}]

It suffices to show (\ref{eqN-0}) for $m>n$.

Now, we consider $H_1\cap \phi(X),\ldots,H_q\cap \phi(X)$, which are located in $m$-subgeneral position with index $\kappa$.

Given $z\in \mathbb{C}$, we arrange so that
$$
\lambda_{H_{1,z}}(\phi\circ f(z))\ge \lambda_{H_{2,z}}(\phi\circ f(z))\ge\cdots\ge \lambda_{H_{m,z}}(\phi\circ f(z))\ge\cdots\ge\lambda_{H_{q,z}}(\phi\circ f(z)),
$$
which implies
\begin{eqnarray}
\sum_{j=1}^q\lambda_{H_j}(\phi\circ f(z))\le \sum_{j=1}^m\lambda_{H_{j,z}}(\phi\circ f(z))+O(1).\label{eqN-2}
\end{eqnarray}

By using the method in \cite{6.5}, we can construct $n$ hyperplanes $\widetilde{H}_{1,z},\ldots,\widetilde{H}_{n,z}$ in $\mathbb{P}^M(\mathbb{C})$ with respect to $H_{1,z}\cap \phi(X),\ldots,H_{m,z}\cap \phi(X)$ such that $\widetilde{H}_{1,z},\ldots,\widetilde{H}_{n,z}$ are in general position on $\phi(X)$.

We start with $\widetilde{H}_{1,z}=H_{1,z}$. For each irreducible component $\Gamma$ of codimension $1$ of $H_{1,z}\cap \phi(X)$, let
\begin{eqnarray*}
V_{\Gamma}=\left\{{\bf{a}}=(a_2,\ldots,a_{m-n+2})\in \mathbb{C}^{m-n+1}\left|\Gamma\subset H_{\bf{a}}\ \mbox{defined by}\ L_{\bf{a}}=\sum_{j=2}^{m-n+2}a_jL_{j,z}\right.\right\}.
\end{eqnarray*}
Where $H_{\bf{a}}=\mathbb{P}^M(\mathbb{C})$ with ${\bf{a}}=(0,\ldots,0)\in \mathbb{C}^{m-n+1}$. By definition, $V_{\Gamma}$ is a subspace of $\mathbb{C}^{m-n+1}$. Since
$$
{\rm{codim}}\bigcap_{j=1}^{m-n+2}(H_{j,z}\cap \phi(X))\ge 2,
$$
there exists $i(2\le i\le m-n+2)$ such that $\Gamma\not\subset H_{i,z}\cap \phi(X)$. This implies that $V_{\Gamma}$ is a proper subspace of $\mathbb{C}^{m-n+1}$. Since the set of irreducible components of $H_{1,z}\cap \phi(X)$ with codimension $1$ is finite,
$$
\mathbb{C}^{m-n+1}\setminus\bigcup_{\Gamma}V_{\Gamma}\neq \emptyset.
$$
Fix $\widetilde{H}_{2,z}$ defined by $\widetilde{L}_{2,z}=\sum\limits_{i=2}^{m-n+2}a_iL_{i,z}$, where ${\bf{a}}=(a_2,\ldots,a_{m-n+2})\in \mathbb{C}^{m-n+1}\setminus\bigcup_{\Gamma}V_{\Gamma}$. Obviously, ${\rm{codim}}(\widetilde{H}_{1,z}\cap\widetilde{H}_{2,z}\cap \phi(X))\ge 2$. For each irreducible component $\Gamma'$ of codimension $2$ of $\widetilde{H}_{1,z}\cap\widetilde{H}_{2,z}\cap \phi(X)$, put
\begin{eqnarray*}
V_{\Gamma'}=\left\{{\bf{a}}=(a_2,\ldots,a_{m-n+3})\in \mathbb{C}^{m-n+2}\left|\Gamma\subset H_{\bf{a}}\ \mbox{defined by}\ L_{\bf{a}}=\sum_{j=2}^{m-n+3}a_jL_{j,z}\right.\right\}.
\end{eqnarray*}
Since
$$
{\rm{codim}}\bigcap_{j=1}^{m-n+3}(H_{j,z}\cap \phi(X))\ge 3,
$$
there exists $i(2\le i\le m-n+3)$ such that $\Gamma'\not\subset H_{i,z}\cap \phi(X)$. This implies that $V_{\Gamma'}$ is a proper subspace of $\mathbb{C}^{m-n+2}$. Since the set of irreducible components of $\widetilde{H}_{1,z}\cap\widetilde{H}_{2,z}\cap \phi(X)$ with codimension $2$ is finite,
$$
\mathbb{C}^{m-n+2}\setminus\bigcup_{\Gamma'}V_{\Gamma'}\neq \emptyset.
$$
Fix $\widetilde{H}_{3,z}$ defined by $\widetilde{L}_{3,z}=\sum\limits_{i=2}^{m-n+3}a_iL_{i,z}$, where ${\bf{a}}=(a_2,\ldots,a_{m-n+3})\in \mathbb{C}^{m-n+1}\setminus\bigcup_{\Gamma'}V_{\Gamma'}$. Hence, ${\rm{codim}}\widetilde{H}_{1,z}\cap\widetilde{H}_{2,z}\cap\widetilde{H}_{3,z}\cap \phi(X)\ge 3$. Repeating the above argument, we obtain $\widetilde{H}_{1,z},\ldots,\widetilde{H}_{n,z}$ which are in general position on $\phi(X)$.

Since there are only finitely many choices of $m$ divisors in $\{H_1\cap \phi(X),\ldots,H_q\cap \phi(X)\}$, we can find a positive constant $C$, independent of $z$, such that, for $i=2,\ldots,n$ and all $z\in \mathbb{C}$,
\begin{eqnarray*}
|\widetilde{L}_{i,z}(\phi\circ f(z))|\le C\max_{2\le j\le m-n+i}|L_{j,z}(\phi\circ f(z))|=C|L_{m-n+i,z}(\phi\circ f(z))|,
\end{eqnarray*}
which implies
\begin{eqnarray*}
\lambda_{H_{m-n+i,z}}(\phi\circ f(z))\le \lambda_{\widetilde{H}_{i,z}}(\phi\circ f(z))+O(1).
\end{eqnarray*}

In summary, we have
\begin{eqnarray}
\sum_{j=1}^m\lambda_{H_{j,z}}(\phi\circ f(z))\le \sum_{i=1}^n\lambda_{\widetilde{H}_{i,z}}(\phi\circ f(z))+\sum_{j=2}^{m-n+1}\lambda_{H_{j,z}}(\phi\circ f(z))+O(1).\label{eqN-3}
\end{eqnarray}
If $\kappa\ge m-n$, then $H_{2,z},\ldots,H_{m-n+1,z}$ are in general position on $\phi(X)$. If $\kappa<m-n$, then $H_{2,z},\ldots,H_{\kappa+1,z}$ are in general position on $\phi(X)$ and
\begin{eqnarray}
\sum_{j=2}^{m-n+1}\lambda_{H_{j,z}}(\phi\circ f(z))\le \frac{m-n}{\kappa}\sum_{j=2}^{\kappa+1}\lambda_{H_{j,z}}(\phi\circ f(z)).\label{eqN-4}
\end{eqnarray}

Note that the union of all hyperplanes $H$ and $\widetilde{H}$ constructed above is a finite set, which may be written as $\{\widetilde{\widetilde{H}}_1,\ldots,\widetilde{\widetilde{H}}_T\}$. By (\ref{eqN-2}), (\ref{eqN-3}) and (\ref{eqN-4}), we have, for every $\varepsilon>0$,
\begin{eqnarray}
\sum_{j=1}^qm_{\phi\circ f}(r,H_j)\le\left(\frac{m-n}{\min\{m-n,\kappa\}}+1\right)\int_0^{2\pi}\max_{\mathcal{K}}\sum_{j\in \mathcal{K}}\lambda_{\widetilde{\widetilde{H}}_j}(\phi\circ f(re^{i\theta}))\frac{d\theta}{2\pi}+O(1),\label{eqN-5}
\end{eqnarray}
where $\max_{\mathcal{K}}$ is taken over all subsets $\mathcal{K}$ of $\{1,\ldots,T\}$ such that $\{\widetilde{\widetilde{H}}_j\}_{j\in \mathcal{K}}$ are in general position on $\phi(X)$.

Now, we directly use the following general form of Theorem A given by Vojta (see also Theorem 22.5 in \cite{12} for the arithmetic case).

\noindent{\bf Theorem E. }(Theorem 22.6 in \cite{12})\quad Let $X\subset \mathbb{P}^M(\mathbb{C})$ be a complex projective variety with $\dim X=n$. Let $D_1,\ldots,D_q$ be hypersurfaces in $\mathbb{P}^M(\mathbb{C})$ whose supports do not contain $X$. Set $d_j=\deg D_j$ for $j=1,\ldots,q$. Let $f:{\mathbb{C}}\rightarrow X$ be an algebraically non-degenerate holomorphic curve. Then, for every $\varepsilon>0$,
\begin{eqnarray*}
\|\int_0^{2\pi}\max_{\mathcal{K}}\sum_{j\in \mathcal{K}}\frac{\lambda_{D_j}(f(re^{i\theta}))}{d_j}\frac{d\theta}{2\pi}\le (n+1+\varepsilon)T_{f}(r),
\end{eqnarray*}
where $\max_{\mathcal{K}}$ is taken over all subsets $\mathcal{K}$ of $\{1,\ldots,q\}$ such that $\{D_j\}_{j\in \mathcal{K}}$ are in general position on $X$.

Now, by (\ref{eqN-5}) and Theorem E, we have
\begin{eqnarray*}
\|\sum_{j=1}^qm_{\phi\circ f}(r,H_j)\le \left(\frac{m-n}{\min\{m-n,\kappa\}}+1\right)(n+1+\varepsilon)T_{\phi\circ f}(r).
\end{eqnarray*}
The desired estimate (\ref{eqN-0}) follows from $T_{\phi\circ f}(r)=T_{f,\widetilde{N}A}(r)+O(1)=\widetilde{N}T_{f,A}(r)+O(1)$, (\ref{eqN-1}) and the above inequality.
\end{proof}

\section{Proof of Theorem \ref{thm1.2}}

\begin{proof}[{\bf Proof of Theorem \ref{thm1.2}}]

Let $N_0$ be a positive integer such that $h^0(\mathcal{L}^N)> 1$ for every integer $N\ge N_0$.

Let $\varepsilon>0$ be given, pick a positive integer $N\ge N_0$ such that
\begin{eqnarray*}
\max_{1\le j\le q}\frac{Nh^0(\mathcal{L}^N)}{\sum_{\alpha=1}^{\infty}h^0(\mathcal{L}^N(-\alpha D_j))}<\max_{1\le j\le q}\gamma(\mathcal{L},D_j)+\frac{\varepsilon}{2},
\end{eqnarray*}
and fix a positive integer $b$ such that
\begin{eqnarray}\label{eeq4.23}
\frac{b+\kappa}{b}\max_{1\le j\le q}\frac{Nh^0(\mathcal{L}^N)}{\sum_{\alpha=1}^{\infty}h^0(\mathcal{L}^N(-\alpha D_j))}<\max_{1\le j\le q}\gamma(\mathcal{L},D_j)+\frac{\varepsilon}{2}.\label{eq3-6}
\end{eqnarray}

Given $z\in \mathbb{C}$, we arrange so that
$$
\lambda_{D_{1,z}}(f(z))\ge \lambda_{D_{2,z}}(f(z))\ge\cdots\ge \lambda_{D_{{\kappa},z}}(f(z))\ge\cdots\ge\lambda_{D_{m,z}}(f(z))\ge\cdots\ge\lambda_{D_{q,z}}(f(z)),
$$
then we have
\begin{eqnarray}
\sum_{j=1}^q\lambda_{D_j}(f(z))\le \sum_{j=1}^m\lambda_{D_{j,z}}(f(z))+O(1)\le\frac{m}{{\kappa}}\sum_{j=1}^{\kappa}\lambda_{D_{j,z}}(f(z))+O(1).\label{eq3-7}
\end{eqnarray}

Note that $D_{1,z},\ldots,D_{{\kappa},z}$ are located in general position and $X$ is smooth. Thus, we can use the filtration constructed in \cite{9,7}. Now we consider the following filtration of $H^0(X,\mathcal{L}^N)$ with respect to $\{D_{1,z},\ldots,D_{{\kappa},z}\}$.

Let $\Delta=\{{\bf a}=(a_i)\in \mathbb{Z}_{\ge 0}^{\kappa}|\sum\limits_{i=1}^{\kappa}a_i=b \}$ where $b$ is the integer in \eqref{eeq4.23}.
For ${\bf a}\in \Delta$ and $x\in \mathbb{R}_{\ge 0}$, let
$$
N({\bf a},x)=\left\{{\bf b}=(b_i)\in \mathbb{Z}_{\ge 0}^{\kappa}\left|\sum_{i=1}^{\kappa}a_ib_i\ge bx\right.\right\}
$$
and
$$
\mathcal{I}({\bf a},x)=\sum_{{\bf b}\in N({\bf a},x)}\mathcal{O}_X\left(-\sum_{i=1}^{\kappa}b_iD_{i,z}\right)
$$
be an ideal of $\mathcal{O}_X$. Set
$$
\mathcal{F}({\bf a})_x=H^0(X,\mathcal{L}^N\otimes\mathcal{I}({\bf a},x))
$$
and
$$
F({\bf a})=\frac{1}{h^0(\mathcal{L}^N)}\int_0^{+\infty}(\dim \mathcal{F}({\bf a})_x)dx.
$$
Then $(\mathcal{F}({\bf a})_x)_{x\in \mathbb{R}_{\ge 0}}$ is a filtration of $H^0(X,\mathcal{L}^N)$ and for any basis $\mathcal{B}_{\bf a}$ of $H^0(X,\mathcal{L}^N)$ with respect to the above filtration $(\mathcal{F}({\bf a})_x)_{x\in \mathbb{R}_{\ge 0}}$, we have
$$
F({\bf a})=\frac{1}{h^0(\mathcal{L}^N)}\sum_{s\in \mathcal{B}_{\bf a}}\mu_{\bf a}(s),
$$
where $\mu_{\bf a}(s)=\sup\{\mu\in \mathbb{R}_{\ge 0}|s\in \mathcal{F}({\bf a})_{\mu}\}$. %Actually, ``$\sup$" can be replaced by ``$\max$". Because, for a fixed ${\bf b}\in \mathbb{Z}_{\ge 0}^{\kappa}$, the set of $x$ with ${\bf b}\in N({\bf a},x)$ is a closed interval $[0,\frac{\sum_{i=1}^{\kappa}a_ib_i}{b}]$ of $\mathbb{R}_{\ge 0}$.  By Theorem 3.6 in \cite{9} (also see Proposition 4.14 in \cite{7}), we know
Since for a fixed ${\bf b}\in \mathbb{Z}_{\ge 0}^{\kappa}$, the set of $x$ with ${\bf b}\in N({\bf a},x)$ is a closed interval of $\mathbb{R}_{\ge 0}$, ``$\sup$" can be replaced by ``$\max$". Hence it follows that
\begin{equation}\label{eqmax}
\mu_{\bf a}(s)=\max\{\mu\in \mathbb{R}_{\ge 0}|s\in \mathcal{F}({\bf a})_{\mu}\}
\end{equation}
By Theorem 3.6 in \cite{9} (also see Proposition 4.14 in \cite{7}), we know
\begin{eqnarray}
F({\bf a})\ge \min_{1\le j\le q}\left(\frac{1}{h^0(\mathcal{L}^N)}\sum_{\alpha=1}^{\infty}h^0(\mathcal{L}^N(-\alpha D_j))\right),\label{eq3-8}
\end{eqnarray}
which implies
\begin{eqnarray}
\sum_{s\in \mathcal{B}_{\bf a}}\mu_{\bf a}(s)\ge \min_{1\le j\le q}\sum_{\alpha=1}^{\infty}h^0(\mathcal{L}^N(-\alpha D_j)).\label{eq3-9}
\end{eqnarray}

We note that any such $s\in \mathcal{B}_{\bf a}$ can be written locally as
$$
s=\sum_{{\bf b}}f_{{\bf b}}\prod_{i=1}^{\kappa}1^{b_i}_{D_{i,z}},
$$
where $f_{{\bf b}}$ is a local section of $\mathcal{L}^N(-\sum\limits_{i=1}^{\kappa}b_iD_{i,z})$, $1_{D_{i,z}}$ is the canonical section of $\mathcal{O}(D_{i,z})$, and the sum is taken for all ${\bf b}\in \mathbb{Z}_{\ge 0}^{\kappa}$ with $\sum\limits_{i=1}^{\kappa}a_ib_i\ge b\mu_{\bf a}(s)$ where we have used \eqref{eqmax}. Actually, $f_{{\bf b}}=0$ for all but finitely many ${\bf b}$.
By a compactness argument, there exists a finite open covering $\{U_j\}_{j\in J_{{\bf a},s}}$ of $X$ and a finite set $K_{{\bf a},s}\subset \mathbb{Z}_{\ge 0}^{\kappa}$ such that
$$
s=\sum_{{\bf b}\in K_{{\bf a},s}}f_{s,j,{\bf b}}\prod_{i=1}^{\kappa}1^{b_i}_{D_{i,z}}
$$
on $U_j$ for all $j\in J_{{\bf a},s}$, where $f_{s,j,{\bf b}}\in \Gamma\left(U_j,\mathcal{L}^N(-\sum\limits_{i=1}^{\kappa}b_iD_{i,z})\right)$ and all ${\bf b}\in K_{{\bf a},s}$ satisfy $\sum\limits_{i=1}^{\kappa}a_ib_i\ge b\mu_{\bf a}(s)$. Hence
%$$
%\lambda_{(s)}(f(z))\ge \min_{{\bf b}\in K_{{\bf a},s}}\sum_{i=1}^{\kappa}b_i\lambda_{D_{i,z}}(f(z))+O(1).
%$$
\begin{equation}\label{eq3-8-1}
\lambda_{(s)}(f(z))\ge \min_{{\bf b}\in K_{{\bf a},s}}\sum_{i=1}^{\kappa}b_i\lambda_{D_{i,z}}(f(z))+O(1).
\end{equation}

Set $t_i=\frac{\lambda_{D_{i,z}}(f(z))}{\sum_{j=1}^{\kappa}\lambda_{D_{j,z}}(f(z))}$, $i=1,\ldots,\kappa$, then $\sum\limits_{i=1}^{\kappa}t_i=1$.
%Let $\beta\ge 1$ be an integer such that $h^0(\mathcal{L}^N(-\beta D_j))=0$ for $j=1,\ldots,q$. Since $s\in \mathcal{B}_{\bf a}$, we may assume that all ${\bf b}\in K_{{\bf a},s}$ satisfy $b_j<\beta$ for all $j$. Now, we fix the integer $b$ with $b\ge \frac{\beta {\kappa}}{N\varepsilon_0}$, where
%\begin{eqnarray}
%0<\varepsilon_0<\frac{\varepsilon}{2\left(\max\limits_{1\le j\le q}\gamma(\mathcal{L},D_j)+\varepsilon \right)\left(\max\limits_{1\le j\le %q}\gamma(\mathcal{L},D_j)+\frac{\varepsilon}{4}\right)}.\label{eq3-10}
%\end{eqnarray}
%We may further assume that the sets $K_{{\bf a},s}$ in \eqref{eq3-8-1} contain no ${\bf b}$ for which $f_{s,j,{\bf b}}=0$ for all $s$ and $j$.
%Therefore, by the choice of the integer $b$, we have the following for all ${\bf b}\in K_{{\bf a},s}$
%\begin{eqnarray}
%\sum_{i=1}^{\kappa}b_i\le {\kappa}\beta\le bN\varepsilon_0.\label{eq3-11}
%\end{eqnarray}
Choose ${\bf a}=(a_i)\in \Delta$ such that $a_i\le (b+\kappa)t_i$ for $i=1,\ldots,{\kappa}$,
i.e.,
\begin{eqnarray*}
\lambda_{(s)}(f(z))&\ge& \min_{{\bf b}\in K_{{\bf a},s}}\sum_{i=1}^{\kappa}b_i\lambda_{D_{i,z}}(f(z))+O(1)\\
&\ge&\left(\sum_{j=1}^{\kappa}\lambda_{D_{j,z}}(f(z))\right)\min_{{\bf b}\in K_{{\bf a},s}}\sum_{i=1}^{\kappa}\frac{b_ia_i}{b+\kappa}+O(1)\\
&\ge&\frac{b\mu_{\bf a}(s)}{b+\kappa}\sum_{j=1}^{\kappa}\lambda_{D_{j,z}}(f(z))+O(1).
\end{eqnarray*}
By (\ref{eq3-7}) and (\ref{eq3-9}), we have
\begin{eqnarray}
\sum_{s\in \mathcal{B}_{\bf a}}\lambda_{(s)}(f(z))&\ge&\left(\frac{b}{b+\kappa}\sum_{s\in \mathcal{B}_{\bf a}}\mu_{\bf a}(s)\right)\sum_{j=1}^{\kappa}\lambda_{D_{j,z}}(f(z))+O(1)\nonumber\\
&\ge&\frac{b}{b+\kappa}\left(\min_{1\le j\le q}\sum_{\alpha=1}^{\infty}h^0(\mathcal{L}^N(-\alpha D_j))\right)\sum_{j=1}^{\kappa}\lambda_{D_{j,z}}(f(z))+O(1)\nonumber\\
&\ge&\frac{b}{b+\kappa}\left(\min_{1\le j\le q}\sum_{\alpha=1}^{\infty}h^0(\mathcal{L}^N(-\alpha D_j))\right)\frac{{\kappa}}{m}\sum_{j=1}^q\lambda_{D_j}(f(z))+O(1).\label{eq3-12}
\end{eqnarray}

Since there are only finitely many choices of $ {\bf a}\in \Delta$ and ${\kappa}$ divisors in $\{D_1,\ldots,D_q\}$, we note that the set of such basis $\mathcal{B}_{\bf a}$ is a finite set, which may be written as $\{\mathcal{B}_{1},\ldots,\mathcal{B}_{T_1}\}$. Set
$$
\mathcal{B}_{1}\cup\cdots\cup\mathcal{B}_{T_1}=\{s_{1},\ldots,s_{T_2}\}.
$$
For each $i=1,\ldots,T_1$, let $J_i\subseteq\{1,\ldots,T_2\}$ be the subset such that $\mathcal{B}_{i}=\{s_j|j\in J_i\}$. Then, by (\ref{eq3-12}),
\begin{eqnarray}
\frac{b}{b+\kappa}\left(\min_{1\le j\le q}\sum_{\alpha=1}^{\infty}h^0(\mathcal{L}^N(-\alpha D_j))\right)\frac{{\kappa}}{m}\sum_{j=1}^q\lambda_{D_j}\le \max_{1\le i\le T_1}\sum_{j\in J_i}\lambda_{(s_j)}+O(1).\label{eq3-13}
\end{eqnarray}
%By Theorem D with $\varepsilon$ taken as $\frac{\varepsilon b \min_{1\le j\le q}\sum_{\alpha=1}^{\infty}h^0(\mathcal{L}^N(-\alpha D_j))}{2(b+\kappa)N}$,
By Theorem D with $\varepsilon$ taken as $\frac{\varepsilon b }{2(b+\kappa)(1+\max_{1\le j\le q}\gamma(\mathcal{L},D_j))}$,
%\begin{eqnarray}
%\|\int_0^{2\pi}\max_{1\le i\le T_1}\sum_{j\in J_i}\lambda_{(s_j)}(f(re^{i\theta}))\frac{d\theta}{2\pi}\le \left(\ell+\frac{\varepsilon b\min\limits_{1\le j\le q}\sum\limits_{\alpha=1}^{\infty}h^0(\mathcal{L}^N(-\alpha D_j))}{2(b+\kappa)N}\right)T_{f,\mathcal{L}^N}(r),\label{eq3-14}
%\end{eqnarray}
{\small \begin{eqnarray}
\|\int_0^{2\pi}\max_{1\le i\le T_1}\sum_{j\in J_i}\lambda_{(s_j)}(f(re^{i\theta}))\frac{d\theta}{2\pi}\le \left(\ell+\frac{\varepsilon b }{2(b+\kappa)(1+\max\limits_{1\le j\le q}\gamma(\mathcal{L},D_j))}\right)T_{f,\mathcal{L}^N}(r),\label{eq3-14}
\end{eqnarray}}
where $\ell=h^0(\mathcal{L}^N)$.

Finally, by (\ref{eq3-6}), (\ref{eq3-13}) and (\ref{eq3-14}), we obtain
%\begin{eqnarray*}
%\|\sum_{j=1}^qm_f(r,D_j)&\le&\frac{m}{{\kappa}}\frac{(b+\kappa)}{b}\frac{\ell+\frac{\varepsilon b\min_{1\le j\le q}\sum_{\alpha=1}^{\infty}h^0(\mathcal{L}^N(-\alpha D_j))}{2(b+\kappa)N}}{\min_{1\le j\le q}\sum_{\alpha=1}^{\infty}h^0(\mathcal{L}^N(-\alpha D_j))}T_{f,\mathcal{L}^N}(r)+O(1)\\
%&\le&\frac{m}{{\kappa}}\frac{(b+\kappa)}{b}\frac{\ell+\frac{\varepsilon b\min_{1\le j\le q}\sum_{\alpha=1}^{\infty}h^0(\mathcal{L}^N(-\alpha D_j))}{2(b+\kappa)N}}{\min_{1\le j\le q}\sum_{\alpha=1}^{\infty}h^0(\mathcal{L}^N(-\alpha D_j))}NT_{f,\mathcal{L}}(r)+O(1)\\
%&\le&\frac{m}{{\kappa}}\left(\max_{1\le j\le q}\gamma(\mathcal{L},D_j)+\varepsilon\right)T_{f,\mathcal{L}}(r)+O(1).
%\end{eqnarray*}

\begin{eqnarray*}
	\|\sum_{j=1}^qm_f(r,D_j)&\le&\frac{m}{{\kappa}}\frac{(b+\kappa)}{b}\frac{\ell+\frac{\varepsilon b }{2(b+\kappa)(1+\max_{1\le j\le q}\gamma(\mathcal{L},D_j))}}{\min_{1\le j\le q}\sum_{\alpha=1}^{\infty}h^0(\mathcal{L}^N(-\alpha D_j))}T_{f,\mathcal{L}^N}(r)+O(1)\\
	&\le&\frac{m}{{\kappa}}\frac{(b+\kappa)}{b}\frac{\ell+\frac{\varepsilon b }{2(b+\kappa)(1+\max_{1\le j\le q}\gamma(\mathcal{L},D_j))}}{\min_{1\le j\le q}\sum_{\alpha=1}^{\infty}h^0(\mathcal{L}^N(-\alpha D_j))}NT_{f,\mathcal{L}}(r)+O(1)\\
	&\le&\frac{m}{{\kappa}}\left(\max_{1\le j\le q}\gamma(\mathcal{L},D_j)+\varepsilon\right)T_{f,\mathcal{L}}(r)+O(1).
\end{eqnarray*}
\end{proof}

\begin{remark}\label{rk3.2}
We assume that $X$ is smooth in Theorem \ref{thm1.2} because we use the filtration which is valid for properly intersecting divisors. In general case, Hussein and Ru \cite{8-3} obtained the following conclusion for $D_1,\ldots,D_q$ being effective Cartier divisors in $m$-subgeneral position ($m\ge n$ and $m>1$) on $X$:
$$
\|m_f(r,D)\le \left(\frac{m(m-1)}{m+n-2}\max_{1\le j\le q}\gamma(\pi^*\mathcal{L},\pi^*D_j)+\varepsilon\right)T_{f,\mathcal{L}}(r),
$$
where $\pi:\widetilde{X}\rightarrow X$ is the normalization of $X$.
Actually, under an additional assumption that $\pi^*D_i$ and $\pi^*D_j$ have no common components for $i\neq j$, we improve Hussein and Ru's result as follows.
\end{remark}

\begin{theorem}\label{thm1.1}
Let $X$ be a complex projective variety, and let $D_1,\ldots,D_q$ be effective Cartier divisors in $m$-subgeneral position ($m>1$) on $X$. Assume that $\pi^*D_i$ and $\pi^*D_j$ have no common components for $i\neq j$, where $\pi:\widetilde{X}\rightarrow X$ is the normalization of $X$. Let $\mathcal{L}$ be a line sheaf on $X$ with $h^0(\mathcal{L}^N)> 1$ for $N$ big enough. Let $f:{\mathbb{C}}\rightarrow X$ be an algebraically non-degenerate holomorphic curve. Then, for every $\varepsilon>0$,
\begin{eqnarray}
\|m_f(r,D)\le \left(\frac{m}{2}\max_{1\le j\le q}\gamma(\pi^*\mathcal{L},\pi^*D_j)+\varepsilon\right)T_{f,\mathcal{L}}(r).\label{eq1-3}
\end{eqnarray}
\end{theorem}

Unlike the above mentioned second main theorems for subgeneral position divisors, (\ref{eq1-3}) is still meaningful if $m<n$. Moreover, the assumption that $\pi^*D_i$ and $\pi^*D_j$ have no common components for $i\neq j$ is automatically satisfied if $m<n$ or $m\ge n$ with $\kappa>1$.

\begin{proof}[{\bf Proof of Theorem \ref{thm1.1}}]
Let $N_0$ be a positive integer such that $h^0(\mathcal{L}^N)> 1$ holds for every integer $N\ge N_0$.

Suppose first that $X$ is normal. Let $\varepsilon>0$ be given, pick a positive integer $N\ge N_0$ such that
\begin{eqnarray}
\max_{1\le j\le q}\frac{Nh^0(\mathcal{L}^N)}{\sum_{\alpha=1}^{\infty}h^0(\mathcal{L}^N(-\alpha D_j))}<\max_{1\le j\le q}\gamma(\mathcal{L},D_j)+\frac{\varepsilon}{2}.\label{eq3-2}
\end{eqnarray}

Given $z\in \mathbb{C}$, we arrange so that
$$
\lambda_{D_{1,z}}(f(z))\ge \lambda_{D_{2,z}}(f(z))\ge\cdots\ge\lambda_{D_{m,z}}(f(z))\ge\cdots\ge\lambda_{D_{q,z}}(f(z)),
$$
then
\begin{eqnarray}
\sum_{j=1}^q\lambda_{D_j}(f(z))\le \sum_{j=1}^m\lambda_{D_{j,z}}(f(z))+O(1)\le \frac{m}{2}(\lambda_{D_{1,z}}(f(z))+\lambda_{D_{2,z}}(f(z)))+O(1).\label{eq3-3}
\end{eqnarray}

For each $D_{\mu}$, consider the following filtration of $H^0(X,\mathcal{L}^N)$:
\begin{eqnarray*}
H^0(X,\mathcal{L}^N)\supset H^0(X,\mathcal{L}^N(-D_{\mu}))\supset H^0(X,\mathcal{L}^N(-2D_{\mu}))\supset\cdots\supset H^0(X,\mathcal{L}^N(-\alpha D_{\mu}))\supset \cdots
\end{eqnarray*}
Let $\mathcal{B}_{\mu}=\{s_1,\ldots,s_\ell\}$ be a basis of $H^0(X,\mathcal{L}^N)$ with respect to the above filtration, where $\ell=h^0(\mathcal{L}^N)$. Write $W_{\alpha}=H^0(X,\mathcal{L}^N(-\alpha D_{\mu}))$, $\alpha\ge 0$, any section $s\in W_{\alpha}\setminus W_{\alpha+1}$ can be written locally as
$$
s=f_{\alpha}1^{\alpha}_{D_{\mu}},
$$
where $f_{\alpha}$ is a local section of $\mathcal{L}^N(-\alpha D_{\mu})$ and $1_{D_{\mu}}$ is the canonical section of $\mathcal{O}(D_{\mu})$.

Hence
\begin{eqnarray*}
\sum_{i=1}^\ell{\rm ord}_{E}(s_{i})&\ge&\left(\sum_{\alpha=1}^{\infty}\alpha \dim W_{\alpha}/W_{\alpha+1}\right){\rm ord}_{E}D_{\mu}= \left(\sum_{\alpha=1}^{\infty}\dim W_{\alpha}\right){\rm ord}_{E}D_{\mu}\\
&=&\left(\sum_{\alpha=1}^{\infty}h^0(\mathcal{L}^N(-\alpha D_{\mu(z)}))\right){\rm ord}_{E}D_{\mu}\\
&\ge& \left(\min_{1\le j\le q}\sum_{\alpha=1}^{\infty}h^0(\mathcal{L}^N(-\alpha D_{j}))\right){\rm ord}_{E}D_{\mu}
\end{eqnarray*}
for any irreducible component $E$ in $D_{\mu}$, i.e.,
$$
\sum_{i=1}^\ell(s_{i})\ge \left(\min_{1\le j\le q}\sum_{\alpha=1}^{\infty}h^0(\mathcal{L}^N(-\alpha D_{j}))\right)D_{\mu}.
$$
For divisors $D_{1,z}$, $D_{2,z}$, we can construct two filtrations of $H^0(X,\mathcal{L}^N)$. By Lemma \ref{lem3.1}, let $\mathcal{B}=\{s_1,\ldots,s_\ell\}$ be the basis of $H^0(X,\mathcal{L}^N)$ with respect to these two filtrations, then we have
$$
\sum_{i=1}^\ell(s_{i})\ge \left(\min_{1\le j\le q}\sum_{\alpha=1}^{\infty}h^0(\mathcal{L}^N(-\alpha D_{j}))\right)D_{1,z}
$$
and
$$
\sum_{i=1}^\ell(s_{i})\ge \left(\min_{1\le j\le q}\sum_{\alpha=1}^{\infty}h^0(\mathcal{L}^N(-\alpha D_{j}))\right)D_{2,z}.
$$
%It follows that
%\begin{eqnarray}
%\sum_{i=1}^\ell(s_{i})\ge \left(\min_{1\le j\le q}\sum_{\alpha=1}^{\infty}h^0(\mathcal{L}^N(-\alpha D_{j}))\right){\rm lcm}(D_{1,z},D_{2,z}),\label{eq3-4}
%\end{eqnarray}
%where ${\rm lcm}(D_{1,z},D_{2,z})=D_{1,z}+D_{2,z}$ because $D_{1,z}$ and $D_{2,z}$ have no common components.

Since $D_{1,z}$ and $D_{2,z}$ have no common components, it follows that
\begin{eqnarray}
\sum_{i=1}^\ell(s_{i})\ge \left(\min_{1\le j\le q}\sum_{\alpha=1}^{\infty}h^0(\mathcal{L}^N(-\alpha D_{j}))\right)(D_{1,z}+D_{2,z}).\label{eq3-4}
\end{eqnarray}

By (\ref{eq3-3}) and (\ref{eq3-4}), we have
\begin{eqnarray*}
\sum_{j=1}^q\lambda_{D_j}(f(z))&\le&\frac{m}{2}(\lambda_{D_{1,z}}(f(z))+\lambda_{D_{2,z}}(f(z)))+O(1)\\
&\le&\frac{m}{2}\frac{1}{\min_{1\le j\le q}\sum_{\alpha=1}^{\infty}h^0(\mathcal{L}^N(-\alpha D_{j}))}\max_{\mathcal{B}}\sum_{s_{i}\in \mathcal{B}}\lambda_{(s_{i})}(f(z))+O(1),
\end{eqnarray*}
where $\max_{\mathcal{B}}$ is taken over all bases $\mathcal{B}$ of $H^0(X,\mathcal{L}^N)$ with respect to $D_{\mu}$ and $D_{\nu}$ with $\mu\neq \nu$.

By using Theorem D with $\varepsilon$ taken as $\varepsilon':=\frac{\varepsilon \min_{1\le j\le q}\sum_{\alpha=1}^{\infty}h^0(\mathcal{L}^N(-\alpha D_{j}))}{2N}$, we obtain
\begin{eqnarray*}
&&\|\sum_{j=1}^qm_f(r,D_j)\\
&\le&\frac{m}{2}\frac{1}{\min_{1\le j\le q}\sum_{\alpha=1}^{\infty}h^0(\mathcal{L}^N(-\alpha D_{j}))}\int_0^{2\pi}\max_{\mathcal{B}}\sum_{s_{i}\in \mathcal{B}}\lambda_{(s_{i})}(f(re^{i\theta}))\frac{d\theta}{2\pi}+O(1)\\
&\le&\frac{m}{2}\frac{1}{\min_{1\le j\le q}\sum_{\alpha=1}^{\infty}h^0(\mathcal{L}^N(-\alpha D_{j}))}(\ell+\varepsilon')T_{f,\mathcal{L}^N}(r)+O(1).
\end{eqnarray*}
Since $T_{f,\mathcal{L}^N}(r)=NT_{f,\mathcal{L}}(r)$ and (\ref{eq3-2}), we have
$$
\|m_f(r,D)\le \frac{m}{2}\left(\max_{1\le j\le q}\gamma(\mathcal{L},D_j)+\varepsilon\right)T_{f,\mathcal{L}}(r)+O(1).
$$

If $X$ is not normal, then we consider the normalization $\pi:\widetilde{X}\rightarrow X$ and divisors $\pi^*D_j$ for all $j$. Notice that
$\pi^*D_j$ are still in $m$-subgeneral position and any two of them have no common components. Hence
$$
\|\sum_{j=1}^qm_{\widetilde{f}}(r,\pi^*D_j)\le \left(\frac{m}{2}\max_{1\le j\le q}\gamma(\pi^*\mathcal{L},\pi^*D_j)+\varepsilon\right)T_{\widetilde{f},\pi^*\mathcal{L}}(r),
$$
where $\widetilde{f}:\mathbb{C}\rightarrow \widetilde{X}$ is the lifting of $f$. As $\pi$ is a birational morphism, by the properties of Weil function and characteristic function, we have
$$
\|\sum_{j=1}^qm_{f}(r,D_j)\le \left(\frac{m}{2}\max_{1\le j\le q}\gamma(\pi^*\mathcal{L},\pi^*D_j)+\varepsilon\right)T_{f,\mathcal{L}}(r).
$$
\end{proof}

\section{Schmidt's subspace theorems}

In this section, we introduce the counterpart in number theory of our main results according to Vojta's dictionary which gives an analogue between Nevanlinna theory and Diophantine approximation (see \cite{11,12}).

Let $k$ be a number field. Denote by $M_k$ the set of places (i.e., equivalence classes of absolute values) of $k$ and write $M_k^\infty$ for the set of archimedean places of $k$.

Let $X$ be a projective variety defined over $k$, let $\mathcal{L}$ be a line sheaf on $X$ and let $D$ be an effective Cartier divisor. For every place $v\in M_k$, we can associate the local Weil functions $\lambda_{\mathcal{L},v}$ and $\lambda_{D,v}$ with respect to $v$, which have similar properties as the Weil function introduced in Section 2. Define
$$
h_{\mathcal{L}}({\bf x})=\sum_{v\in M_k}\lambda_{\mathcal{L},v}({\bf x})\ \mbox{for}\ {\bf x}\in X
$$
and
$$
m_S({\bf x}, D)=\sum_{v\in S}\lambda_{D,v}({\bf x})\ \mbox{for}\ {\bf x}\in X\setminus {\rm Supp}D,
$$
where $S$ is a finite subset of $M_k$ containing $M_k^{\infty}$.

Instead of Theorem D, we shall use the following general form of Schmidt's subspace theorem given by Ru and Vojta \cite{7}.

\noindent{\bf Theorem F.} (Theorem 2.7 in \cite{7})\quad \textit{Let $X$ be a projective variety defined over $k$, and let $\mathcal{L}$ be a line sheaf on $X$. Let $V$ be a linear subspace of $H^0(X,\mathcal{L})$ with $\dim V>1$, and let $s_1,\ldots,s_q$ be nonzero elements of $V$. For each $j=1,\ldots,q$, let $D_j$ be the Cartier divisor $(s_j)$. Let $S$ be a finite subset of $M_k$ containing $M_k^{\infty}$, let $\varepsilon>0$ and $c\in \mathbb{R}$. Then there is a proper Zariski-closed subset $Z$ of $X$ such that
\begin{eqnarray*}
\sum_{v\in S}\max_{\mathcal{K}}\sum_{j\in \mathcal{K}}\lambda_{D_j,v}({\bf x})\le (\dim V+\varepsilon)h_{\mathcal{L}}({\bf x})+c
\end{eqnarray*}
holds for all ${\bf x}\in (X\setminus Z)(k)$. Here $\max_{\mathcal{K}}$ is taken over all subsets $\mathcal{K}$ of $\{1,\ldots,q\}$ such that the sections $\{s_j\}_{j\in \mathcal{K}}$ are linearly independent.}

Now, we state the counterparts of Theorems \ref{thm1.4}---\ref{thm1.2}, whose proofs are similar and is therefore omitted here.

\begin{theorem}\label{thm5.1}
	Let $X$ be a projective variety of dimension $n$ defined over $k$, and let $D_1,\ldots,D_q$ be effective Cartier divisors in $m$-subgeneral position with index $\kappa(>1)$ on $X$. Suppose that there exists an ample divisor $A$ on $X$ and positive integers $d_j$ such that $D_j\sim d_jA$ for $j=1,\ldots,q$. Let $S$ be a finite subset of $M_k$ containing $M_k^{\infty}$. Then, for every $\varepsilon>0$,
	 \begin{eqnarray*}
	 	\sum_{j=1}^q\frac{1}{d_j}m_S({\bf x},D_j)\le \left(\max\left\{\frac{m}{2\kappa},1\right\}(n+1)+\varepsilon\right)h_{\mathcal{O}(A)}({\bf x})
	 \end{eqnarray*}
	 holds for all $k$-rational points outside a proper Zariski-closed subset.
\end{theorem}

%Using Theorem \ref{thm5.1}, one can discuss the general Thue's equations, and easily give an improvement of Theorems 2 and 2' in \cite{10}, in which the complete intersection condition is avoided.
Theorem \ref{thm5.1} gives an improvement of Theorems 2 and 2' in \cite{10}, moreover, the complete intersection condition for $X$ is dropped.

\begin{theorem}
	Let $X$ be a projective variety of dimension $n$ defined over $k$, and let $D_1,\ldots,D_q$ be effective Cartier divisors in $m$-subgeneral position with index $\kappa$ on $X$. Suppose that there exists an ample divisor $A$ on $X$ and positive integers $d_j$ such that $D_j\sim d_jA$ for $j=1,\ldots,q$. Let $S$ be a finite subset of $M_k$ containing $M_k^{\infty}$. Then, for every $\varepsilon>0$,
	\begin{eqnarray*}
		\sum_{j=1}^q\frac{1}{d_j}m_S({\bf x},D_j)\le \left(\left(\frac{m-n}{\max\{1,\min\{m-n,\kappa\}\}}+1\right)(n+1)+\varepsilon\right)h_{\mathcal{O}(A)}({\bf x})
	\end{eqnarray*}
	holds for all $k$-rational points outside a proper Zariski-closed subset of $X$.
\end{theorem}

\begin{theorem}
Let $X$ be a smooth projective variety defined over $k$, and let $D_1,\ldots,D_q$ be effective divisors in $m$-subgeneral position with index $\kappa$ on $X$. Let $D=D_1+\cdots+D_q$. Let $\mathcal{L}$ be a line sheaf on $X$ with $h^0(\mathcal{L}^N)> 1$ for $N$ big enough. Let $S$ be a finite subset of $M_k$ containing $M_k^{\infty}$. Then, for every $\varepsilon>0$,
\begin{eqnarray*}
	m_S({\bf x},D)\le \left(\frac{m}{\kappa}\max_{1\le j\le q}\gamma(\mathcal{L},D_j)+\varepsilon\right)h_{\mathcal{L}}({\bf x})
\end{eqnarray*}
holds for all $k$-rational points outside a proper Zariski-closed subset of $X$.	
\end{theorem}

%\end{CJK*}
\end{document}